\documentclass[11pt]{article}

\usepackage{geometry}
\usepackage[inline]{showlabels}
\usepackage{xr-hyper}
\usepackage{times,enumitem}
\usepackage{hyperref}
\hypersetup{
    colorlinks = true,
    citecolor= blue
}
\usepackage{natbib}
\usepackage{url}
\usepackage{amsfonts,amsmath,amssymb,amsthm}
\usepackage{dsfont}
\usepackage{wrapfig}
\usepackage{verbatim,float,url}
\usepackage{graphicx,psfrag}
\usepackage{subcaption}
\usepackage[english]{babel}
\usepackage{cancel}
\usepackage{color}
\usepackage{pifont}

\usepackage{algorithm}
\usepackage{algpseudocode}
\usepackage{pifont}

\usepackage{mathtools}

\newcommand{\argmin}{\mathop{\mathrm{argmin}}}

\newtheorem{theorem}{Theorem}
\newtheorem{lemma}[theorem]{Lemma}

\usepackage{mdwlist}	
\theoremstyle{remark}

\theoremstyle{definition}

\newtheorem{example}{Example}

\usepackage{tikz}  
\usetikzlibrary{trees}

\def\E{\mathbb{E}}
\def\P{\mathbb{P}}
\def\one{\mathds{1}}

\def\half{\frac{1}{2}}

\def\row{\mathrm{row}}

\def\nuli{\mathrm{nullity}}
\def\spa{\mathrm{span}}

\def\htheta{\hat{\theta}}
\def\ttheta{\tilde{\theta}}

\def\R{\mathbb{R}}

\def\cB{\mathcal{B}}

\def\cI{\mathcal{I}}

\def\cP{\mathcal{P}}
\def\cS{\mathcal{S}}

\def\BV{\mathrm{BV}}
\def\BD{\mathrm{BD}}
\def\PC{\mathrm{PC}}

\graphicspath{{figs/}}

\usepackage{mathtools}
\DeclarePairedDelimiter\ceil{\lceil}{\rceil}
\DeclarePairedDelimiter\floor{\lfloor}{\rfloor}

\begin{document}

\title{The DFS Fused Lasso: Linear-Time Denoising over General Graphs} 

\author{
Oscar Hernan Madrid Padilla$^{1}$ \\
{\tt oscar.madrid@utexas.edu} \\
\and
James Sharpnack$^{2}$ \\
{\tt jsharpna@gmail.com} \\ 
\and
James Scott$^{1,3}$ \\
{\tt james.scott@mccombs.utexas.edu} \\
\and
Ryan J. Tibshirani$^{4,5}$ \\
{\tt ryantibs@stat.cmu.edu} \\ 
\and
\begin{small}
\begin{tabular}{c}
  $^{1}$ Department of Statistics and Data Sciences, University of
  Texas, Austin, TX 78712 \\ 
  $^{2}$ Department of Statistics, University of California at Davis,
  Davis, CA 95616 \\  
  $^{3}$ McCombs School of Business, University of Texas, Austin, TX
  78712 \\ 
  $^{4}$ Department of Statistics, Carnegie Mellon University,
  Pittsburgh, PA 15213 \\
  $^{5}$ Machine Learning Department, Carnegie Mellon University,
  Pittsburgh, PA 15213
  \end{tabular}
\end{small}
}

\date{}

\maketitle

\begin{abstract}
The fused lasso, also known as (anisotropic) total variation
denoising, is widely used for piecewise constant signal estimation
with respect to a given undirected graph. The fused lasso estimate is
highly nontrivial to compute when the underlying graph is large and
has an arbitrary structure.  But for a special graph structure,
namely, the chain graph, the fused lasso---or simply, 1d fused
lasso---can be computed in linear time. In this paper, we establish a
surprising connection between    
the total variation of a generic signal defined over an arbitrary
graph, and the total variation of this signal over a chain graph
induced by running depth-first search (DFS) over the nodes of the   
graph. Specifically, we prove that for any signal, its total variation over
the induced chain graph is no more than twice its total variation over
the original graph. This connection leads to several interesting
theoretical and computational conclusions.  Denoting by $m$ and $n$
the number of edges and nodes, respectively, of the graph in question,
our result implies that for an underlying signal with total variation
$t$ over the graph, the fused lasso achieves a mean squared error rate
of  \smash{$t^{2/3} n^{-2/3}$}. Moreover, precisely
the same mean squared error rate is achieved by running the 1d
fused lasso on the induced chain graph from running DFS. Importantly,
the latter estimator is simple and computationally cheap,
requiring only $O(m)$ operations for constructing the DFS-induced chain
and $O(n)$ operations for computing the 1d fused lasso solution
over this chain.  Further, for trees that have bounded max degree, the
error rate of \smash{$t^{2/3} n^{-2/3}$} cannot be improved,
in the sense that it is the minimax rate for 
signals that have total variation $t$ over the tree. Finally, we
establish several related results---for example, a similar result
for a roughness measure defined by the $\ell_0$ norm of differences
across edges in place of the the total variation metric.

\end{abstract}

\section{Introduction}

We study the graph denoising problem, i.e., estimation of a signal
$\theta_0 \in \R^n$ from noisy data 
\begin{equation}
\label{eq:model}
y_i = \theta_{0,i} + \epsilon_i, \quad i=1,\ldots,n,
\end{equation}
when the components of $\theta_0$ are
associated with the vertices of an undirected, connected
graph $G = (V,E)$.   Without a loss of generality, we denote   
$V=\{1,\ldots,n\}$. Versions of this problem arise in diverse  
areas of science and engineering, such as gene expression analysis,
protein mass spectrometry, and image denoising.  The problem  
is also archetypal of numerous internet-scale machine learning  
tasks that involve propagating labels or information across edges
in a network (e.g., a network of users, web pages, or YouTube videos).    

Methods for graph denoising have been studied extensively in
machine learning and signal processing. 
In machine learning, graph kernels have been proposed for
classification and regression, in both supervised and semi-supervised
data settings (e.g.,
\citet{belkin2002using,smola2003kernels,zhu2003semisupervised,
zhou2005learning}). 
In signal processing, a considerable focus has been placed on the 
construction of wavelets over graphs (e.g., 
\cite{crovella2003graph,coifman2006diffusion,gavish2010multiscale,
hammond2011wavelets,sharpnack2013detecting,shuman2013emerging}).   
We will focus our study on the {\it fused lasso} over graphs, also
known as (anisotropic) {\it total variation} denoising over
graphs.  
Proposed by \citet{rudin1992nonlinear} in the signal processing
literature, and \citet{tibshirani2005sparsity} in the statistics
literature, the fused lasso estimate is defined by the solution of a
convex optimization problem,  
\begin{equation}
\label{eq:fused_lasso}
\htheta_G = \argmin_{\theta \in \R^n} \; 
\half \|y - \theta\|_2^2 + \lambda \| \nabla_G \theta\|_1, 
 \end{equation}
where $y=(y_1,\ldots,y_n) \in \R^n$ the vector of observed data,
$\lambda \geq 0$ is a tuning parameter, and $\nabla_G \in \R^{m \times
  n}$ is the edge incidence matrix of the graph $G$.  Note that the
subscript on the incidence matrix $\nabla_G$ and the fused lasso
solution \smash{$\htheta_G$} in \eqref{eq:fused_lasso} emphasize that 
these quantities are defined with respect to the graph $G$.
The edge incidence matrix $\nabla_G$ can be defined
as follows, using some notation and terminology from algebraic graph 
theory (e.g., \cite{godsil2013algebraic}). First, we assign an
arbitrary orientation to edges in the graph, i.e., for each edge $e
\in E$, we arbitrarily select one of the two joined vertices to be the
head, denoted $e^+$, and the other to be the tail, denoted
$e^-$. Then, we define a row $(\nabla_G)_e$ of $\nabla_G$,
corresponding to the edge $e$, by   
\begin{equation*}
(\nabla_G)_{e,e^+}=1, \;\; (\nabla_G)_{e,e^-}=-1, \;\;
(\nabla_G)_{e,v}=0 \;\,\text{for all $v \not= e^+,e^-$},
\end{equation*}
for each $e \in E$.  Hence, for an arbitrary $\theta \in \R^n$, we
have
\begin{equation*}
\|\nabla_G \theta\|_1 = \sum_{e \in E} |\theta_{e^+} - \theta_{e^-}|.
\end{equation*}
We can see that the particular choice of orientation does not
affect the value \smash{$\|\nabla_G \theta\|_1$}, which we refer to as 
the {\it total variation} of $\theta$ over the graph $G$. 

\subsection{Summary of results}
\label{sec:summary}

We will wait until Section \ref{sec:related} to give a detailed 
review of literature, both computational and theoretical, on
the fused lasso. Here we simply highlight a key
computational aspect of the fused lasso to motivate the main results
in our paper.  The fused lasso solution in \eqref{eq:fused_lasso}, for
a graph $G$ of arbitrary structure, is highly nontrivial to compute.
For a chain graph, however, the fused lasso solution can be computed
in linear time (e.g., using dynamic programming or specialized 
taut-string methods).  

The question we answer is: how can we use this fact to our
advantage, when seeking to solve \eqref{eq:fused_lasso} over an
arbitrary graph?  Given a generic graph structure $G$ that has $m$
edges and $n$ nodes, it is obvious that we can define a chain graph by
running depth-first search (DFS) over the nodes.  Far less obvious is
that, for any signal, its total variation over the DFS-induced chain
graph never exceeds twice its total variation over the original graph.
This fact, which we prove, has the following three notable
consequences (the first being computational, and next two
statistical).   

\begin{enumerate}
\item No matter the structure of $G$, we can denoise any signal
  defined over this graph in $O(m+n)$ operations: $O(m)$
  operations for DFS and $O(n)$ operations for the 1d fused lasso on 
  the induced chain.  We call the corresponding estimator---the 1d
  fused lasso run on the DFS-induced chain---the {\it DFS fused
    lasso}. 

\item For an underlying signal $\theta_0$ that generates the data, as   
  in \eqref{eq:model}, such that \smash{$\theta_0 \in
    \BV_G(t)$}, where \smash{$\BV_G(t)$} is the class of signals
  with total variation at most $t$, defined in
  \eqref{eq:bv_class}, the DFS fused lasso estimator has mean squared
  error (MSE) on the order of \smash{$t^{2/3}n^{-2/3}$}. 

\item For an underlying signal \smash{$\theta_0 \in \BV_G(t)$}, the
  fused lasso estimator over the original graph, in
  \eqref{eq:fused_lasso}, also has MSE on the order of
  \smash{$t^{2/3}n^{-2/3}$}.   
\end{enumerate}

The fact that such a fast rate,
\smash{$t^{2/3}n^{-2/3}$}, applies for the fused lasso estimator over 
{\it any} connected graph structure is somewhat surprising.  It
implies that the chain graph represents the hardest graph structure
for denoising signals of bounded variation---at least, hardest for the 
fused lasso, since as we have shown, error rates on general connected
graphs can be no worse than the chain rate of
\smash{$t^{2/3}n^{-2/3}$}. 

We also complement these MSE upper bounds with the following minimax
lower bound over trees.     

\begin{enumerate}
\item[4.] When $G$ is a tree of bounded max degree, the minimax
  MSE over the class \smash{$\BV_G(t)$} scales at
  the rate \smash{$t^{2/3} n^{-2/3}$}. Hence, in this setting, the DFS 
  fused lasso estimator attains the optimal rate, as does the fused 
  lasso estimator over $G$. 
\end{enumerate}

Lastly, we prove the following for signals with a
bounded number of nonzero edge differences.

\begin{enumerate}
\item[5.] For an underlying signal \smash{$\theta_0 \in \BD_G(s)$},  
  where \smash{$\BD_G(s)$} is the class of signals with at most   
  $s$ nonzero edge differences, defined in \eqref{eq:bd_class}, the
  DFS fused lasso (under a condition on the spacing of nonzero
  differences over the DFS-induced chain) has MSE on the order of   
  \smash{$s(\log{s}+\log\log{n}) \log{n}/n+ s^{3/2}/n$}. When $G$ is a  
  tree, the minimax MSE over the class \smash{$\BD_G(s)$} scales as
  $s\log(n/s)/n$. Thus, in this setting, the DFS fused lasso estimator
  is only off by a $\log\log{n}$ factor provided that $s$ is small.
\end{enumerate}

This DFS fused lasso gives us an $O(n)$ time algorithm for
nearly minimax rate-optimal denoising over trees.  On paper, this
only saves a factor of $O(\log{n})$ operations, as recent work (to be  
described in Section \ref{sec:related}) has produced an $O(n\log{n})$
time algorithm for the fused lasso over trees, by
extending earlier dynamic programming ideas over chains.  However,
dynamic programming on a tree is (a) much more complex 
than dynamic programming on a chain (since it relies on  
sophisticated data structures), and (b) noticeably slower in practice 
than dynamic programming over a chain, especially for large problem
sizes. Hence there is still a meaningful difference---both in terms of  
simplicity and practical computational efficiency---between the DFS
fused lasso estimator and the fused lasso over a generic tree.

For a general graph structure, we cannot claim that the statistical 
rates attained by the DFS fused lasso estimator are optimal, nor can
we claim that they match those of fused lasso over the original 
graph. As an example, recent work (to be discussed in 
Section \ref{sec:related}) studying the 
fused lasso over grid graphs shows that estimation error rates for 
this problem can be much faster than those attained by the DFS fused
lasso (and thus the minimax rates over trees).  What should be 
emphasized, however, is that the DFS fused lasso can still be a
practically useful method for any graph, running in linear time
(in the number of edges) no matter the graph structure, a scaling that
is beneficial for truly large problem sizes.  

\subsection{Assumptions and notation}
\label{sec:assumptions}

Our theory will be primarily phrased in terms of the mean squared
error (MSE) an estimator \smash{$\htheta$} of the mean parameter
$\theta_0$ in \eqref{eq:model}, assuming that
$\epsilon=(\epsilon_1,\ldots,\epsilon_n)$ has i.i.d.\ mean zero 
sub-Gaussian components, i.e.,  
\begin{equation}
\label{eq:sub_gauss}
\E(\epsilon_i)=0, \quad \text{and} \quad \P(|\epsilon_i|>t) \leq M 
\exp\big(-t^2/(2\sigma^2)\big), \;\, \text{all}\;\, t \geq 0, \quad   
\text{for}\;\, i=1,\ldots,n,
\end{equation}
for constants $M,\sigma>0$. The MSE of \smash{$\htheta$} will be 
denoted, with a slight abuse of notation, by
\begin{equation*}
\| \htheta - \theta_0 \|_n^2 = \frac{1}{n} \| \htheta - \theta_0 \|_2^2.  
\end{equation*}
(In general, for a vector $x \in \R^n$, we denote its scaled $\ell_2$
norm by \smash{$\|x\|_n = \|x\|_2/\sqrt{n}$}.)  Of course, the MSE
will depend not only on the estimator \smash{$\htheta$} in question
but also on the assumptions that we make about $\theta_0$. We will
focus our study on two classes of signals.  The first is the {\it
  bounded variation class}, defined with respect to the graph $G$,
and a radius parameter $t>0$, as  
\begin{equation}
\label{eq:bv_class}
\BV_G(t) = \{ \theta \in \R^n : \| \nabla_G \theta \|_1 \leq t\}.
\end{equation}
The second is the {\it bounded differences class}, defined again with 
respect to the graph $G$, and a now a sparsity parameter $s>0$, as
\begin{equation}
\label{eq:bd_class}
\BD_G(s) = \{ \theta \in \R^n : \| \nabla_G \theta \|_0 \leq s\}. 
\end{equation}
We call measure of roughness used in the bounded differences class the   
{\it cut metric}, given by replacing the $\ell_1$ norm used to define
the total variation metric by the $\ell_0$ norm, i.e.,    
\begin{equation*}
\|\nabla_G \theta\|_0 = \sum_{e \in E} 1\{\theta_{e^+} \not=
\theta_{e^-}\}, 
\end{equation*}
which counts the number of nonzero edge differences that appear 
in $\theta$. Hence, we may think of the former class in
\eqref{eq:bv_class} 
as representing a type of weak sparsity across these edge differences,
and the latter class in \eqref{eq:bd_class} as representing a type of 
strong sparsity in edge differences. 

When dealing with the chain graph, on $n$ vertices, we will use the
following modifications to our notation. We write 
\smash{$\nabla_{\mathrm{1d}} \in \R^{(n-1)\times n}$} for the edge
incidence matrix of the chain, 
i.e.,   
\begin{equation}
\label{eq:nabla_chain}
\nabla_{\mathrm{1d}} = \left[\begin{array}{rrrrr} 
-1 & 1 & 0 & \ldots &  0\\
0 & -1 & 1 & \ldots &  0\\
\vdots &  & \ddots & \ddots & \\ 
0 & 0 & \ldots & -1 & 1 
\end{array}\right].
\end{equation}
We also write \smash{$\htheta_{\mathrm{1d}}$} for the solution of the
fused lasso problem in \eqref{eq:fused_lasso} over the chain, also
called the {\it 1d fused lasso} solution, i.e., to be explicit, 
\begin{equation}
\label{eq:fused_lasso_1d}
\htheta_{\mathrm{1d}} = \argmin_{\theta \in \R^n} \; 
\half \|y - \theta\|_2^2 + \lambda \sum_{i=1}^{n-1} 
|\theta_{i+1} - \theta_i|.
 \end{equation}
We write \smash{$\BV_{\mathrm{1d}}(t)$} and
\smash{$\BD_{\mathrm{1d}}(s)$} for the bounded variation and bounded
differences classes with respect to the chain, i.e., to be explicit, 
\begin{align*}
\BV_{\mathrm{1d}}(t) &= \{ \theta \in \R^n : \| \nabla_{\mathrm{1d}}
\theta \|_1 \leq t\}, \\
\BD_{\mathrm{1d}}(s) &= \{ \theta \in \R^n : \| \nabla_{\mathrm{1d}}  
\theta \|_0 \leq s\}.  
\end{align*}

Lastly, in addition to the standard notation $a_n=O(b_n)$, for 
sequences $a_n,b_n$ such that
$a_n/b_n$ is upper bounded for $n$ large enough, we use
$a_n \asymp b_n$ to denote that both $a_n=O(b_n)$ and
$a^{-1}_n=O(b_n^{-1})$.  Also, for random sequences $A_n,B_n$, we use 
$A_n = O_\P(B_n)$ to denote that $A_n/B_n$ is bounded in probability.  

\subsection{Related work}
\label{sec:related}

Since its inception in the signal processing and statistics
communities in \citet{rudin1992nonlinear} and
\citet{tibshirani2005sparsity}, respectively, there has been
an impressive amount of work on total variation penalization and the 
fused lasso. We do not attempt to give a complete coverage, 
but point out some relevant computational and theoretical advances,
covering the two categories separately.

\paragraph{Computational.}
On the computational side, it is first worth pointing out that there
are multiple efficient algorithms for solving the fused
lasso problem over a chain graph, i.e., the 1d fused lasso problem.     
\citet{davies2001local} derived an algorithm based on a ``taut
string'' perspective that solves the 1d fused lasso problem in $O(n)$
time (but, the fact that their taut string method solves the 1d
fused lasso problem was not explicitly stated in the work). 
This was later extended by \citet{condat2012direct,barbero2014modular}
to allow for arbitrary weights in both of the individual penalty and
loss terms. \citet{johnson2013dynamic} proposed an entirely different 
$O(n)$ time algorithm for the fused lasso based on dynamic
programming. The taut string and dynamic programming algorithms 
are extremely fast in practice (e.g., they can solve a 1d fused lasso
problem with $n$ in the tens of millions in just a few seconds on a
standard laptop).   

\citet{kolmogorov2016total} extended the dynamic programming approach 
of \citet{johnson2013dynamic} to solve the fused lasso problem on a
tree. Their algorithm in theoretically very efficient, with
$O(n\log{n})$ running time, but the implementation that achieves this
running time (we have found) can be practically slow for large problem  
sizes, compared to dynamic programming on a chain graph.  Alternative
implementations are possible, and may well improve practical
efficiency, but as far as we see it, they will all involve
somewhat sophisticated data structures in the ``merge'' steps in the
forward pass of dynamic programming.

\citet{barbero2014modular} extended (though not in the same direct 
manner) fast 1d fused lasso optimizers to work over grid graphs, using
operator splitting techniques like Douglas-Rachford splitting. Their
techniques appear to be quite efficient in practice, and the authors 
provide thorough comparisons and a thorough literature review of
related methods. Over general graphs structures, many algorithms
have been proposed, 
e.g., to highlight a few: \citet{chambolle2009total} described a direct
algorithm based on a reduction to parametric max flow programming;
\citet{hoefling2010path,tibshirani2011solution} gave solution path
algorithms (tracing out the solution in \eqref{eq:fused_lasso} over
all $\lambda \in [0,\infty]$); \citet{chambolle2011first} described
what can be seen as a kind of preconditioned ADMM-style algorithm; 
\citet{kovac2011nonparametric} described an active set approach; 
\citet{tansey2015fast} leveraged fast 1d fused lasso solvers in an
ADMM decomposition over trails of the graph; most recently, 
\citet{landrieu2016cut} derived a new method based on graph cuts.
We emphasize that, even with the advent of these numerous clever
computational techniques for the fused lasso over general graphs, it
is still far slower to solve the fused lasso over an arbitrary graph
than it is to solve the fused lasso over a chain. 

\paragraph{Theoretical.} On the theoretical side, it seems that the
majority of statistical theory on the fused lasso can be placed into
two categories: analysis of changepoint recovery, and analysis of MSE.
Some examples of works focusing on changepoint recovery are 
\citet{rinaldo2009properties,harchaoui2010multiple,qian2012pattern,
rojas2014change}.  The statistical theory 
will concern MSE rates, and hence we give a more detailed review of 
related literature for this topic. 

We begin with results for chain graphs. \citet{mammen1997locally}
proved, when \smash{$\theta_0 \in \BV_{\mathrm{1d}}(t)$}, that the
1d fused lasso estimator estimator \smash{$\htheta_{\mathrm{1d}}$}
with \smash{$\lambda \asymp t^{-1/3}n^{1/3}$} satisfies 
\begin{equation}
\label{eq:bv_bound_1d}
\|\htheta_{\mathrm{1d}} - \theta_0\|_n^2 = O_\P(t^{2/3} n^{-2/3}). 
\end{equation}
This is indeed the minimax MSE rate for the class
\smash{$\BV_{\mathrm{1d}}(t)$}, as implied by the minimax
results in \citet{donoho1998minimax}.   
(For descriptions of the above upper bound and this minimax rate    
in a language more in line with that of the current paper, see 
\citet{tibshirani2014adaptive}.) Recently,
\citet{lin2016approximate} improved on earlier results for
the bounded differences class in \citet{dalalyan2014prediction}, 
and proved that when \smash{$\theta_0 \in \BD_{\mathrm{1d}}(s)$},  
the 1d fused lasso estimator \smash{$\htheta_{\mathrm{1d}}$}   
with \smash{$\lambda \asymp (nW_n)^{1/4}$} satisfies  
\begin{equation}
\label{eq:bd_bound_1d}
\|\htheta_{\mathrm{1d}} - \theta_0\|_n^2 = O_\P\bigg(\frac{s}{n} 
\Big((\log{s}+\log\log{n})\log{n} + \sqrt{n/W_n}\Big)\bigg), 
\end{equation}
where $W_n$ denotes the minimum distance between positions at which
nonzero differences occur in $\theta_0$, more precisely, 
\smash{$W_n= \min \{ |i-j| : (\nabla_{\mathrm{1d}} \theta_0)_i\not=0,
(\nabla_{\mathrm{1d}} \theta_0)_j\not=0 \}$}.  When these nonzero
differences or ``jumps'' in $\theta_0$ are evenly spaced apart, we
have $W_n \asymp n/s$, and the above becomes, for 
\smash{$\lambda \asymp \sqrt{n}s^{-1/4}$},
\begin{equation}
\label{eq:bd_bound_1d_no_wn}
\|\htheta_{\mathrm{1d}} - \theta_0\|_n^2 =
O_\P\bigg(\frac{s(\log{s}+\log\log{n})\log{n}}{n} + 
\frac{s^{3/2}}{n} \bigg). 
\end{equation}
This is quite close to the minimax lower bound, whose rate is 
\smash{$s \log(n/s) /n$}, that we establish for the class
\smash{$\BD_{\mathrm{1d}}(s)$}, in Theorem \ref{thm:bd_minimax}. (The  
minimax lower bound that we prove this theorem actually holds beyond 
the chain graph, and applies to tree graphs).  We can
see that the 1d fused lasso rate in \eqref{eq:bd_bound_1d_no_wn} is
only off by a factor of $\log\log{n}$, provided that $s$ does not grow
too fast (specifically, \smash{$s=O((\log{n}\log\log{n})^2)$}).  

Beyond chain graphs, the story is in general much less clear, however,  
interesting results are known in special cases. For a
$d$-dimensional grid graph, with $d \geq 2$, \citet{hutter2016optimal}  
recently improved 
on results of \citet{wang2016trend}, showing that for 
\smash{$\theta_0 \in \BV_G(t) \cap \BD_G(s)$}, the fused lasso 
estimator \smash{$\htheta_G$} over $G$ satisfies
\begin{equation}
\label{eq:grid_bound}
\|\htheta_G - \theta_0\|_n^2 = O_\P\bigg(\min\{t,s\}
\frac{\log^a{n}}{n} \bigg).
\end{equation}
when \smash{$\lambda \asymp \log^{a/2}{n}$},
where $a=2$ if $d=2$, and $a=1$ if $d \geq 3$.
A minimax lower bound on the MSE rate for the \smash{$\BV_G(t)$} 
class over a grid $G$ of dimension $d \geq 2$ was established
to be \smash{$t \sqrt{\log(n/t)}/n$}, by \citet{sadhanala2016total}.
This makes the rate achieved by the fused lasso in
\eqref{eq:grid_bound} nearly optimal for bounded variation signals,
off by at most a \smash{$\log^{3/2}{n}$} factor when $d=2$, and a
\smash{$\log{n}$} factor when $d \geq 3$.

\citet{wang2016trend,hutter2016optimal} also derived MSE rates for the 
fused lasso over several other graph structures, such as
Erdos-Renyi random graphs, Ramanujan $d$-regular graphs, star graphs,
and complete graphs.  As it is perhaps the most relevant to our
goals in this paper, we highlight the MSE bound from
\citet{wang2016trend} that applies to arbitrary connected 
graphs.  Their Theorem 3 implies, for 
a generic connected graph $G$, \smash{$\theta_0 \in \BV_G(t)$},
that the fused lasso estimator \smash{$\htheta_G$} over $G$ with  
\smash{$\lambda \asymp \sqrt{n\log{n}}$} satisfies
\begin{equation}
\label{eq:gen_bound}
\|\htheta_G - \theta_0\|_n^2 = O_\P\bigg( 
t \sqrt{\frac{\log{n}}{n}} \bigg). 
\end{equation}
(See Appendix \ref{app:gen_bound} for details.)  In Theorem
\ref{thm:bv_bound}, we show that the universal $tn^{-1/2}$ rate
(ignoring log terms) in \eqref{eq:gen_bound} for the fused lasso over
an arbitrary connected graph can be improved to
\smash{$t^{2/3}n^{-2/3}$}. In Theorem
\ref{thm:bv_bound_dfs}, we show that the same rate can indeed be 
achieved by a simple, linear-time algorithm: the DFS fused lasso.   

\subsection{Outline}
\label{sec:outline} 

In Section \ref{sec:dfs}, we prove a simple but key lemma relating the
$\ell_1$ norm (and $\ell_0$) norm of differences on a tree and a chain
induced by running DFS.  We then define the DFS fused lasso
estimator. In Section \ref{sec:bv}, we derive MSE rates for the DFS
fused lasso, and the fused lasso over the original graph $G$ in
question, for signals of bounded variation.  We also derive lower  
bounds for the minimax MSE rate over trees.  In Section \ref{sec:bd},
we proceed similarly, but for signals with bounded differences.
In Section \ref{sec:experiments}, we cover numerical experiments, and  
in Section \ref{sec:discussion}, we summarize our work
and also describe some potential extensions.

\section{The DFS fused lasso}
\label{sec:dfs}

In this section, we define the DFS-induced chain graph and the DFS
fused lasso.

\subsection{Tree and chain embeddings}
\label{sec:dfs_lemma} 

We start by studying some of the fundamental properties associated
with total variation on general graphs, and embedded trees and chains.
Given a graph $G=(V,E)$, let $T=(V,E_T)$ be an arbitrary spanning tree
of $G$.  It is clear that for any signal, its total variation of over
$T$ is no larger than its total variation over $G$, 
\begin{equation}
\label{eq:tree_l1}
\| \nabla_T \theta \|_1 = \sum_{e \in E_T} 
|\theta_{e^+}- \theta_{e^-}| \leq \sum_{e \in E} 
|\theta_{e^+}- \theta_{e^-}| = \| \nabla_G \theta\|_1,
\quad \text{for all} \;\, \theta \in \R^n.
\end{equation}
The above inequality, albeit very simple, reveals to us the following
important fact: if the underlying mean $\theta_0$ in \eqref{eq:model}
is assumed to be smooth with respect to the graph $G$, inasmuch as  
\smash{$\|\nabla_G \theta_0\|_1 \leq t$}, then it must also be smooth 
with respect to any spanning tree $T$ of $G$, since
\smash{$\|\nabla_T \theta_0\|_1 \leq t$}.  Roughly 
speaking, computing the fused lasso solution in \eqref{eq:fused_lasso}
over a spanning tree $T$, instead of $G$, would therefore still be
reasonable for the denoising purposes, as the mean $\theta_0$
would still be smooth over $T$ according to the total variation
metric. 

The same property as in \eqref{eq:tree_l0} also holds if we
replace total variation by the cut metric:
\begin{equation}
\label{eq:tree_l0}
\| \nabla_T \theta \|_0 = \sum_{e \in E_T} 
1\{\theta_{e^+} \not= \theta_{e^-}\} \leq \sum_{e \in E}   
1\{\theta_{e^+} \not= \theta_{e^-}\} = \| \nabla_G \theta\|_0,
\quad \text{for all} \;\, \theta \in \R^n.
\end{equation}
Thus for the mean $\theta_0$, the property \smash{$\|\nabla_G
  \theta_0\|_0 \leq s$} again implies \smash{$\|\nabla_T \theta_0\|_0 
  \leq s$} for any spanning tree $T$ of $G$, and this would again
justify solving the fused lasso over $T$, in place of $G$, 
assuming smoothness of $\theta_0$ with respect to the cut metric in
the first place.  

Here we go one step further than \eqref{eq:tree_l1},
\eqref{eq:tree_l0}, and assert that analogous properties actually hold 
for specially embedded chain graphs. The next lemma gives the key
result. 

\begin{lemma}
\label{lem:dfs}
Let $G=(V,E)$ be a connected graph, where recall we write  
$V=\{1,\ldots,n\}$. Consider depth-first search (DFS) run on $G$, 
and denote by $v_1,\ldots,v_n$ the nodes in the order in which they
are reached by DFS. Hence, DFS first visits $v_1$, then $v_2$, then
$v_3$, etc. This induces a bijection $\tau : \{1,\ldots,n\} \to 
\{1,\ldots,n\}$, such that 
\begin{equation*}
\tau(i) = v_i, \quad \text{for all} \;\, i=1,\ldots,n.
\end{equation*}
Let $P \in \R^{n \times n}$ denote the permutation associated with
$\tau$. Then it holds that
\begin{equation}
\label{eq:dfs_l1}
\|\nabla_{\mathrm{1d}} P \theta\|_1 \leq 2 \|\nabla_G \theta\|_1, 
\quad \text{for all} \;\, \theta \in \R^n,
\end{equation}
as well as 
\begin{equation}
\label{eq:dfs_l0}
\|\nabla_{\mathrm{1d}} P \theta\|_0 \leq 2 \|\nabla_G \theta\|_0,
\quad \text{for all} \;\, \theta \in \R^n.
\end{equation}
\end{lemma} 

\begin{proof}
The proof is simple. Observe that
\begin{equation}
\label{eq:dfs_sum}
\|\nabla_{\mathrm{1d}} P \theta\|_1 = \sum_{i=1,\ldots,n-1}
| \theta_{\tau(i+1)} - \theta_{\tau(i)} |,
\end{equation}
and consider an arbitrary summand \smash{$| \theta_{\tau(i+1)} - 
  \theta_{\tau(i)} |$}. There are now two cases to examine. First,
suppose $\tau(i)$ is not a leaf node, and $\tau(i+1)$ has not yet  
been visited by DFS; then there is an edge $e \in E$ such that     
\smash{$\{e^-,e^+\} = \{\tau(i), \tau(i+1)\}$}, and 
\smash{$|\theta_{\tau(i+1)} - \theta_{\tau(i)} | = |\theta_{e^+} -  
  \theta_{e^-}|$}. Second, suppose that either $\tau(i)$ is a leaf 
node, or all of its neighbors have already been visited by DFS; then
there is a path $p = \{p_1,\ldots,p_r\}$ in the graph such that
$p_1=\tau(i)$, $p_r=\tau(i+1)$, and each $\{p_j,p_{j+1}\} \in E$,
$j=1,\ldots,r-1$, so that by the triangle inequality
\begin{equation*}
|\theta_{\tau(i+1)} - \theta_{\tau(i)} | \leq
  \sum_{j=1}^{r-1} |\theta_{p_{j-1}} - \theta_{p_j}|.
\end{equation*}
Applying this logic over all terms in the sum in \eqref{eq:dfs_sum},
and invoking the fundamental property that DFS visits each edge
exactly twice (e.g., Chapter 22 of \citet{cormen2001introduction}), we
have established \eqref{eq:dfs_l1}. The proof for
\eqref{eq:dfs_l0} follows from precisely the same arguments.  
\end{proof}

\begin{example}
The proof behind Lemma \ref{lem:dfs} can also be clearly demonstrated 
through an example. We consider $G$ to be a binary tree graph with $n=7$
nodes, shown below, where we have labeled the nodes according to the
order in which they are visited by DFS (i.e., so that here $P$ is the
identity).   

\begin{center}
  \begin{tikzpicture}[level distance=1.5cm,
    level 1/.style={sibling distance=5cm},
    level 2/.style={sibling distance=1.5cm}]
    \node {1}
    child {node {2}
      child {node {3}}
      child {node {4}}
    }
    child {node {5}
      child {node {6}}
      child {node {7}}
    };
  \end{tikzpicture} 
\end{center}
In this case, 
\begin{align*}
\|\Delta_{\mathrm{1d}} \theta\|_1 &= 
\sum_{i=1}^6 |\theta_{i+1}-\theta_i| \\
&\leq |\theta_2-\theta_1| + |\theta_3-\theta_2| + 
\big( |\theta_3-\theta_2| + |\theta_4-\theta_2| \big) + 
\big( |\theta_4-\theta_2| + |\theta_2-\theta_1| +
|\theta_5-\theta_1| \big) \\
&\quad + |\theta_6 - \theta_5| + \big( |\theta_6-\theta_5| + 
|\theta_7 - \theta_5| \big) \\
&\leq 2 \sum_{e \in G} |\theta_{e^+} - \theta_{e^-}| 
=2 \|\nabla_G \theta\|_1,
\end{align*}
where in the inequality above, we have used triangle inequality for
each term in parentheses individually.
\end{example}

\subsection{The DFS fused lasso}
\label{sec:dfs_estimator}

We define the DFS fused lasso estimator, 
\smash{$\htheta_{\mathrm{DFS}}$}, to be the fused lasso estimator  
over the chain graph induced by running DFS on $G$. Formally, if
$\tau$ denotes the bijection associated with the DFS ordering (as
described 
in Lemma \ref{lem:dfs}), then the DFS-induced chain graph can be
expressed as $C=(V,E_C)$ where $V=\{1,\ldots,n\}$ and  
$E_C=\{\{\tau(1),\tau(2)\}, \ldots, \{\tau(n-1),\tau(n)\}\}$.
Denoting by $P$ the permutation matrix associated with $\tau$, the   
edge incidence matrix of $C$ is simply 
\smash{$\nabla_C = \nabla_{\mathrm{1d}} P$}, and the DFS fused lasso
estimator is given by
\begin{align}
\nonumber
\htheta_{\mathrm{DFS}} &\;=\; \argmin_{\theta \in \R^n} \;  
\half \|y - \theta\|_2^2 + 
\lambda \|\nabla_{\mathrm{1d}} P \theta\|_1 \\ 
\label{eq:fused_lasso_dfs}
&\;=\; P^\top \left( \argmin_{\theta \in \R^n} \; 
\half \|Py - \theta\|_2^2 + \lambda \sum_{i=1}^{n-1} 
|\theta_{i+1} - \theta_i| \right).
\end{align}
Therefore, we only need to compute the 1d fused lasso estimator on a 
permuted data vector $Py$, and apply the inverse permutation operator 
$P^\top$, in order to compute \smash{$\htheta_{\mathrm{DFS}}$}.  

Given the permutation matrix $P$, the computational cost of
\eqref{eq:fused_lasso_dfs} is $O(n)$, since, to recall the
discussion in Section \ref{sec:related}, the 1d fused lasso problem 
\eqref{eq:fused_lasso_1d} can be solved in $O(n)$ operations with
dynamic programming or taut string algorithms. The permutation $P$ is
obtained by running DFS, which requires $O(m)$ operations,
and makes the total computation cost of the DFS fused lasso estimator 
$O(m+n)$. 

It should be noted that, when multiple estimates are desired over the 
same graph $G$, we must only run DFS once, and all subsequent
estimates on the induced chain require just $O(n)$ operations.

The bounds in \eqref{eq:dfs_l1}, \eqref{eq:dfs_l0} for the DFS chain
are like those in \eqref{eq:tree_l1}, \eqref{eq:tree_l0} for spanning
trees, and carry the same motivation as that discussed above
for spanning trees, beneath \eqref{eq:tree_l1}, \eqref{eq:tree_l0}: if
the mean $\theta_0$ is assumed to be smooth with respect to $t$,
insofar as its total variation satisfies \smash{$\|\nabla_G
  \theta_0\|_1 \leq t$}, then denoising with respect to $C$ would also
be reasonable, in that \smash{$\|\nabla_{\mathrm{1d}} P \theta_0\|_1
  \leq 2t$}; the same can be said for the cut metric.  However, it is
the rapid $O(m+n)$ computational cost of the DFS fused lasso, and also   
the simplicity of the dynamic programming and taut string algorithms
for the 1d fused lasso problem \eqref{eq:fused_lasso_1d}, that makes 
\eqref{eq:dfs_l1}, \eqref{eq:dfs_l0} particularly appealing
compared to   
\eqref{eq:tree_l1}, \eqref{eq:tree_l0}.  To recall the discussion in
Section \ref{sec:related}, the fused lasso can in principle be
computed efficiently over a tree, in $O(n\log{n})$ operations using 
dynamic programming, but this requires a much more cumbersome
implementation and in practice we have found it to be noticeably  
slower. 

\subsection{Running DFS on a spanning tree}
\label{sec:dfs_remarks}

We can think of the induced chain graph, as described in the last
section, as being computed in two steps:   
\begin{enumerate}
\item[(i)] run DFS to compute a spanning tree $T$ of $G$;
\item[(ii)] run DFS on the spanning tree $T$ to define the chain $C$.  
\end{enumerate}
Clearly, this is the same as running DFS on $G$ to define the induced  
chain $C$, so decomposing this process into two steps as we have done 
above may seem odd.  But this decomposition provides a useful 
perspective because it leads to the idea that we could compute the
spanning tree $T$ in Step (i) in any fashion, and then proceed with
DFS on $T$ in Step 2 in order to define the 
chain $C$.  Indeed, any spanning tree in Step (i) will lead to a chain
$C$ that has the properties \eqref{eq:dfs_l1}, \eqref{eq:dfs_l0}
as guaranteed by Lemma \ref{lem:dfs}. This may be of interest if we
could compute a spanning tree $T$ that better represents the topology
of the original graph $G$, so that the differences over the eventual
chain $C$ better mimicks those over $G$. 

An example of a spanning tree whose
topology is designed to reflect that of the original graph is a
low-stretch spanning tree. Current interest on low-stretch
spanning trees began with the breakthrough results in
\citet{elkin2008lower}; most recently, \citet{abraham2012petal} 
showed that a spanning tree with average stretch
\smash{$O(\log{n}\log\log{n})$} can be computed in
\smash{$O(m\log{n}\log\log{n})$} operations.
In Section \ref{sec:experiments}, we investigate low-stretch 
spanning trees experimentally.

In Section \ref{sec:weighted}, we discuss a setting in which the
fused lasso problem \eqref{eq:fused_lasso} has arbitrary penalty
weights, which gives rise to a weighted graph $G$. 
In this setting, an example of a spanning tree that can be crafted
so that its edges represent important differences in the original
graph is a maximum spanning tree. Prim's and Kruskal's minimum
spanning tree algorithms, each of which take $O(m\log{n})$ time
\citep{cormen2001introduction}, can be used to compute a maximum
spanning tree after we negate all edge weights.  

\subsection{Averaging multiple DFS estimators}

Notice that several DFS-induced chains can be formed from a single 
seed graph $G$, by running DFS itself on $G$ with different
random starts (or random decisions about which edge to follow at   
each step in DFS), or by computing different spanning trees $T$ of $G$   
(possibly themselves randomized) on which we run DFS, or by some 
combination, etc. Denoting by 
\smash{$\htheta^{(1)}_{\mathrm{DFS}}, \htheta^{(2)}_{\mathrm{DFS}},
  \ldots, \htheta^{(K)}_{\mathrm{DFS}}$} the DFS fused lasso estimators
fit to $K$ different induced chains, we might believe that the average  
estimator, \smash{$(1/K) \sum_{k=1}^K \htheta^{(k)}_{\mathrm{DFS}}$}, 
will have good denoising performance, as it incorporates fusion at 
each node in multiple directions. In Section \ref{sec:experiments}, we
demonstrate that this intuition holds true (at least, across the set
of experiments we consider).

\section{Analysis for signals of bounded variation}
\label{sec:bv}

Throughout this section, we assume that the underlying mean $\theta_0$ 
in \eqref{eq:model} satisfies \smash{$\theta_0 \in \BV_G(t)$} for
a generic connected graph $G$.
We derive upper bounds on the MSE rates of the DFS fused lasso and the  
fused lasso over $G$. We also derive a tight lower bound on the
minimax MSE when $G$ is a tree that of bounded degree.

\subsection{The DFS fused lasso}

The analysis for the DFS fused lasso estimator is rather 
straightforward. By assumption, \smash{$\|\nabla_G \theta_0\|_1 \leq 
  t$}, and thus \smash{$\|\nabla_{\mathrm{1d}} P \theta_0\|_1 \leq 
  2t$} by \eqref{eq:dfs_l1} in Lemma \ref{lem:dfs}.   Hence, we may  
think of our model \eqref{eq:model} as giving us i.i.d.\ data  
$Py$ around \smash{$P\theta_0 \in \BV_{\mathrm{1d}}(2t)$}, and we may
apply existing results from \citet{mammen1997locally} on the 1d fused 
lasso for bounded variation signals, as described in
\eqref{eq:bv_bound_1d} in Section \ref{sec:related}. This establishes the
following.

\begin{theorem}
\label{thm:bv_bound_dfs}
Consider a data model \eqref{eq:model}, with i.i.d.\ 
sub-Gaussian errors as in \eqref{eq:sub_gauss}, and
\smash{$\theta_0 \in \BV_G(t)$}, where $G$ is a generic
connected graph. Then for any DFS ordering of $G$ yielding a
permutation matrix $P$, the DFS fused lasso estimator 
\smash{$\htheta_{\mathrm{DFS}}$} in \eqref{eq:fused_lasso_dfs}, with
a choice of tuning parameter \smash{$\lambda \asymp t^{-1/3}n^{1/3}$},
has MSE converging in probability at the rate
\begin{equation}
\label{eq:bv_bound_dfs}
\|\htheta_{\mathrm{DFS}} - \theta_0\|_n^2 = O_\P(t^{2/3} n^{-2/3}).
\end{equation}
\end{theorem}

We note that, if multiple DFS fused lasso estimators
\smash{$\htheta^{(1)}_{\mathrm{DFS}}, \htheta^{(2)}_{\mathrm{DFS}},
  \ldots, \htheta^{(K)}_{\mathrm{DFS}}$} are computed across multiple 
different DFS-induced chains on $G$, then the average estimator
clearly satisfies the same bound as in \eqref{eq:bv_bound_dfs},
\begin{equation*}
\left\|\frac{1}{K}\sum_{k=1}^K \htheta_{\mathrm{DFS}}^{(k)} - 
  \theta_0\right\|_n^2 = O_\P(t^{2/3} n^{-2/3}),
\end{equation*}
provided that $K$ is held constant, by the triangle inequality.  

\subsection{The graph fused lasso}

Interestingly, the chain embedding result \eqref{eq:dfs_l1} in Lemma
\ref{lem:dfs} is not only helpful for establishing the MSE rate for the
DFS fused lasso estimator in Theorem \ref{thm:bv_bound_dfs}, but
it can also be used to improve the best known rate for the
original fused lasso estimator over the graph $G$.   
In Section \ref{sec:related}, we
described a result \eqref{eq:gen_bound} that follows from
\citet{wang2016trend},  
establishing an MSE rate of \smash{$tn^{-1/2}$} rate (ignoring log
terms) for the fused lasso estimator over a connected graph $G$, when
\smash{$\|\nabla_G \theta_0\|_1 \leq t$}.  In fact, as we will now
show, this can be improved to a rate of \smash{$t^{2/3}n^{-2/3}$},
just as in \eqref{eq:bv_bound_dfs} for the DFS fused lasso. 

\citet{wang2016trend} present a framework for deriving
fast MSE rates for fused lasso estimators based on entropy. They
show in their Lemma 9 that a bound in probability on the sub-Gaussian
complexity    
\begin{equation}
\label{eq:sg_complexity}
\max_{x \in \cS_G(1)} \; \frac{\epsilon^\top x}{\|x\|_2^{1-w/2}},
\end{equation}
for some $0<w<2$, where \smash{$\cS_G(1) = \{ x \in \row(\nabla_G) :
  \|\nabla_G x \|_1 
  \leq 1\}$}, leads to a bound in probability on the MSE of the fused
lasso estimator \smash{$\htheta_G$} over $G$.  (\citet{wang2016trend}
actually assume  
Gaussian errors, but their Lemma 9, Theorem 10, Lemma 11, and 
Corollary 12 still hold for sub-Gaussian errors as in
\eqref{eq:sub_gauss}). The sub-Gaussian complexity in
\eqref{eq:sg_complexity} is typically controlled via an entropy bound 
on the class \smash{$\cS_G(1)$}.  Typically, one thinks of controlling
entropy by focusing on specific classes of graph structures $G$.
Perhaps surprisingly, Lemma \ref{lem:dfs} shows we can uniformly
control the sub-Gaussian complexity \eqref{eq:sg_complexity} over
all connected graphs. 

For any DFS-induced chain $C$ constructed from $G$, note 
first that \smash{$\row(\nabla_G) = \spa\{\one\}^\perp = 
  \row(\nabla_C)$}, where $\one=(1,\ldots,1) \in \R^n$ is the 
vector of all 1s.  This, and \eqref{eq:dfs_l1} in Lemma
\ref{lem:dfs}, imply that
\begin{equation*}
\max_{x \in \cS_G(1)} \; \frac{\epsilon^\top x}{\|x\|_2^{1-w/2}} \leq
\max_{x \in \cS_C(2)} \; \frac{\epsilon^\top x}{\|x\|_2^{1-w/2}}.
\end{equation*}
Now, taking $w=1$,
\begin{equation*}
\max_{x \in \cS_C(2)} \; \frac{\epsilon^\top x}{\|x\|_2^{1/2}} =  
\max_{\substack{x \;:\; \one^\top x = 0, \\
\|\nabla_{\mathrm{1d}} P x \|_1 \leq 2}} \; 
\frac{\epsilon^\top x}{\|x\|_2^{1/2}} = 
\max_{\substack{x \;:\; \one^\top x = 0, \\ 
\|\nabla_{\mathrm{1d}} x \|_1 \leq 1}} \; 
\frac{2^{-1/2} (P\epsilon)^\top x}{\|x\|_2^{1/2}} =    
O_\P(n^{1/4}).
\end{equation*}
The last step (asserting that the penultimate term is
\smash{$O_\P(n^{1/4})$}) holds   
by first noting that $P\epsilon$ is equal in law to $\epsilon$ (as we
have assumed i.i.d.\ components of the error vector), and then 
applying results on the chain graph in Theorem 10, Lemma 11, and
Corollary 12 of \citet{wang2016trend}.  Applying Lemma 9 of
\citet{wang2016trend}, we have now established the following result. 

\begin{theorem}
\label{thm:bv_bound}
Consider a data model \eqref{eq:model}, with i.i.d.\ 
sub-Gaussian errors as in \eqref{eq:sub_gauss}, and
\smash{$\theta_0 \in \BV_G(t)$}, where $G$ is a generic
connected graph. Then the fused lasso estimator
\smash{$\htheta_G$} over $G$, in \eqref{eq:fused_lasso}, under 
a choice of tuning parameter \smash{$\lambda \asymp t^{-1/3}n^{1/3}$}, 
has MSE converging in probability at the rate 
\begin{equation}
\label{eq:bv_bound}
\|\htheta_G - \theta_0\|_n^2 = O_\P(t^{2/3} n^{-2/3}).  
\end{equation}
\end{theorem}

In a sense, the above theorem suggests that the chain graph is
among the hardest graphs for denoising bounded variation signals,
since the fused lasso estimator on any connected graph $G$
will achieve an MSE rate in that is at least as good as in the chain 
rate, if not better. In this vein, it is worth emphasizing that the
MSE bound in \eqref{eq:bv_bound} is not tight for certain graph
structures; a good example is the 2d grid, where we must compare 
\eqref{eq:bv_bound} from the theorem to the known MSE bound
in \eqref{eq:grid_bound} from \citet{hutter2016optimal}, the latter
being only log factors from optimal, as shown in
\citet{sadhanala2016total}. It is natural for the 2d grid graph 
to consider the scaling \smash{$t \asymp \sqrt{n}$} (as argued in 
\citet{sadhanala2016total}), in which case the rates for the fused 
lasso estimator are \smash{$n^{-1/3}$} from Theorem \ref{thm:bv_bound} 
versus \smash{$(\log^2{n}) n^{-1/2}$} from \citet{hutter2016optimal}.

\subsection{Minimax lower bound over trees}

We derive a lower bound for the MSE over the class \smash{$\BV_G(t)$}
when $G$ is a tree graph. The proof applies Assouad's Lemma
\citep{yu1997assouad}, over a  
discrete set of probability measures constructed by a careful
partitioning of the vertices of $G$, that balances both the sizes of
each partition element and the number of edges crossing in between
partition elements. It is deferred until Appendx \ref{app:bv_minimax}.  

\begin{theorem}
\label{thm:bv_minimax}
Consider a data model \eqref{eq:model}, with i.i.d.\ 
Gaussian errors $\epsilon_i \sim N(0,\sigma^2)$, $i=1,\ldots,n$, and 
with \smash{$\theta_0 \in \BV_G(t)$}, where $G$ is a tree graph,
having maximum degree \smash{$d_{\mathrm{max}}$}.  Then there exists 
absolute constants $N,C > 0$, such that for 
\smash{$n / (t d_{\mathrm{max}}) > N$},    
\begin{equation}
\label{eq:bv_minimax}
\inf_{\htheta} \; \sup_{\theta_0 \in \BV_G(t)} \;
\E \| \htheta - \theta_0 \|_n^2  \geq C 
\bigg( \frac{t}{\sigma d_{\mathrm{max}}^2 n} \bigg)^{2/3}. 
\end{equation}
\end{theorem}

The theorem demonstrates that, for trees of bounded degree, such as
the chain and balanced $d$-ary trees, the fused lasso estimator over 
the tree achieves achieves the minimax rate, as does the DFS fused
lasso.  

\section{Analysis for signals with bounded differences} 
\label{sec:bd}

We assume that the underlying mean $\theta_0$ 
in \eqref{eq:model} satisfies \smash{$\theta_0 \in \BD_G(s)$} for 
a generic connected graph $G$.  We analyze the MSE of the DFS fused
lasso, as well as (a particular formulation of) wavelet denoising over
$G$. We again establish a lower bound on the minimax MSE when $G$  
is a tree. 

\subsection{The DFS fused lasso}

As it was for the bounded variation case, the analysis for the DFS
fused lasso estimator is straightforward. By assumption,
\smash{$\|\nabla_G \theta_0\|_0 \leq  
  s$}, thus \smash{$\|\nabla_{\mathrm{1d}} P \theta_0\|_0 \leq  
  2s$} by \eqref{eq:dfs_l0} in Lemma \ref{lem:dfs}, and we may   
think of our model \eqref{eq:model} as having i.i.d.\ data  
$Py$ around \smash{$P\theta_0 \in \BD_{\mathrm{1d}}(2s)$}.  Applying
an existing result on the 1d fused lasso for bounded differences 
signals, as described in \eqref{eq:bd_bound_1d}, from
\citet{lin2016approximate}, gives the following result.

\begin{theorem}
\label{thm:bd_bound_dfs}
Consider a data model \eqref{eq:model}, with i.i.d.\ 
sub-Gaussian errors as in \eqref{eq:sub_gauss}, and
\smash{$\theta_0 \in \BD_G(s)$}, for a connected graph $G$.
Consider an arbitrary DFS ordering of $G$,
that defines a permutation matrix $P$ and the DFS fused lasso
estimator \smash{$\htheta_{\mathrm{DFS}}$} in
\eqref{eq:fused_lasso_dfs}. 
Denote by \smash{$W_n= \min \{ |i-j| : (\nabla_{\mathrm{1d}} P
  \theta_0)_i\not=0, (\nabla_{\mathrm{1d}} P \theta_0)_j\not=0 \}$}
the minimum distance between positions, measured along 
the DFS-induced chain, at which nonzero 
differences or jumps occur in $\theta_0$.  Then, under 
a choice of tuning parameter \smash{$\lambda \asymp (nW_n)^{1/4}$}, 
the DFS fused lasso estimator has MSE converging in probability at the 
rate 
\begin{equation}
\label{eq:bd_bound_dfs}
\|\htheta_{\mathrm{DFS}} - \theta_0\|_n^2 = O_\P\bigg(\frac{s}{n} 
\Big((\log{s}+\log\log{n})\log{n} + \sqrt{n/W_n}\Big)\bigg).
\end{equation}
Hence, if the $s$ jumps along the DFS chain are evenly spaced
apart, i.e., $W_n \asymp n/s$, then for
\smash{$\lambda \asymp \sqrt{n}s^{-1/4}$},
\begin{equation}
\label{eq:bd_bound_dfs_no_wn} 
\|\htheta_{\mathrm{DFSr}} - \theta_0\|_n^2 = O_\P\bigg(
\frac{s(\log{s}+\log\log{n})\log{n}}{n} + 
\frac{s^{3/2}}{n} \bigg). 
\end{equation}
\end{theorem}

An undesirable feature of applying existing 1d
fused lasso results for signals with bounded differences, in the above 
result, is the dependence on $W_n$ in the DFS fused lasso error bound  
\eqref{eq:bd_bound_dfs} (we applied the  result \eqref{eq:bd_bound_1d}   
from \citet{lin2016approximate}, but the bounds from
\citet{dalalyan2014prediction} also depend on $W_n$, and as far as we 
can tell, so should any analysis of the 1d fused lasso for signals
with bounded differences). In the 1d setting, assuming that $W_n
\asymp n/s$, which says that jumps in $\theta_0$ occur at roughly
equally 
spaced positions, is fairly reasonable; but to assume the same when
the jumps are measured with respect to the DFS-induced chain, as we
must in order to establish \eqref{eq:bd_bound_dfs_no_wn}, is perhaps
not. Even if the differences apparent in $\theta_0$ over edges in $G$
are somehow (loosely speaking) spaced far apart, running
DFS could well produce an ordering such that jumps in $P\theta_0$
occur at positions very close together.
We reiterate that the MSE bounds for the DFS fused lasso for
bounded variation signals, in Theorem \ref{thm:bv_bound_dfs}, do not
suffer from any such complications.

\subsection{Graph wavelet denoising}

We compare the performances of the DFS fused lasso and wavelet
denoising using spanning tree wavelets, for
signals with bounded differences. For spanning tree wavelets, the 
construction starts with a spanning tree and carefully defines a 
hierarchical  
decomposition by recursively finding and splitting around a balancing
vertex, which is a vertex whose adjacent subtrees are of size
at most half of the original tree; this decomposition is used to
construct an unbalanced Haar wavelet basis, as in
\citet{singh2010detecting}. In \citet{sharpnack2013detecting}, it was
shown that for any connected graph $G$, the constructed wavelet basis 
$W \in \R^{n\times n}$ satisfies    
\begin{equation}
\label{eq:wavelet_l0}
\|W \theta \|_0 \le \ceil*{\log d_{\mathrm{max}}}  
\ceil*{\log n} \| \nabla_G \theta \|_0,  
\quad \text{for all} \;\, \theta \in \R^n,
\end{equation}
where \smash{$d_{\mathrm{max}}$} is the maximum degree of $G$, and the
above holds regardless of choice of spanning tree in the wavelet
construction. Now consider the wavelet denoising estimator 
\begin{equation}
\label{eq:wavelet}
\htheta_W = \argmin_{\theta \in \R^n} \; \half  \|y - \theta\|_2^2 + 
\lambda\|W \theta\|_1.   
\end{equation}
The following is an immediate consequence of \eqref{eq:wavelet_l0},
the fact that the wavelet basis $W$ is orthonormal, and standard
results about soft-thresholding (e.g., Lemma 2.8 in
\citep{johnstone2011gaussian}).    
 
\begin{theorem}
\label{thm:bd_bound_wav}
Consider a data model \eqref{eq:model}, with i.i.d.\ 
Gaussian errors $\epsilon_i \sim N(0,\sigma^2)$, $i=1,\ldots,n$, and 
with \smash{$\theta_0 \in \BV_G(t)$}, where $G$ is a connected graph, 
having maximum degree \smash{$d_{\mathrm{max}}$}. Then the spanning
tree wavelet estimator \smash{$\htheta_W$} in \eqref{eq:wavelet}, with
a choice \smash{$\lambda \asymp \sqrt{\log{n}}$}, has MSE converging
in expectation at the rate
 \begin{equation}
\label{eq:bd_bound_wav} 
\E \|\htheta_W - \theta_0\|_n^2 = O 
\bigg( \frac{s \log d_{\mathrm{max}}\log^2{n}}{n} \bigg).
\end{equation}
\end{theorem}

The result in \eqref{eq:bd_bound_wav} has the advantage over the DFS
fused lasso result in \eqref{eq:bd_bound_dfs} that it does not depend
on a hard-to-interpret quantity like $W_n$, the minimum spacing
between jumps along the DFS-induced chain.  But when (say) 
\smash{$d_{\mathrm{max}} \asymp 1$}, $s \asymp 1$, and we are willing
to assume that $W_n \asymp n$ (meaning the jumps of $\theta_0$ occur
at positions evenly spaced apart on the DFS chain), we can see that
the spanning tree wavelet rate in \eqref{eq:bd_bound_wav} is just
slightly slower than the DFS fused lasso rate in
\eqref{eq:bd_bound_dfs_no_wn}, by a factor of
\smash{$\log{n}/\log\log{n}$}.  

While the comparison between the DFS fused lasso and wavelet rates, 
\eqref{eq:bd_bound_dfs} and \eqref{eq:bd_bound_wav}, show
an advantage to spanning tree wavelet denoising, as it does not
require assumptions about the spacings between nonzero differences in 
$\theta_0$, we have found nonetheless that the DFS fused lasso to
performs well in practice compared to spanning tree wavelets, and
indeed often outperforms the latter in terms of MSE.  Experiments 
comparing the two methods are presented Section
\ref{sec:experiments}. 

\subsection{Minimax lower bound for trees}
\label{sec:bd_minimax}

We now derive a lower bound for the MSE over the class
\smash{$\BD_G(s)$} when $G$ is a tree graph. The proof relates the
current denoising problem to one of estimating sparse normal means,
with a careful construction of the sparsity set using degree
properties of trees. It is deferred until Appendix
\ref{app:bd_minimax}.   

\begin{theorem}
\label{thm:bd_minimax}
Consider a data model \eqref{eq:model}, with i.i.d.\ 
Gaussian errors \smash{$\epsilon_i \sim N(0,\sigma^2)$},
$i=1,\ldots,n$, and with \smash{$\theta_0 \in \BD_G(s)$}, where $G$ is
a tree. Then there are absolute constants $N,C > 0$, such that for $n
/ s > N$,    
\begin{equation}
\label{eq:bd_minimax}
\inf_{\htheta} \; \sup_{\theta_0 \in \BD_G(s)} \;
\E \| \htheta - \theta_0 \|_n^2 \geq C \sigma^2
\frac{s}{n} \log \Big( \frac{n}{s} \Big).
\end{equation}
\end{theorem}

The MSE lower bound in \eqref{eq:bd_minimax} shows that, when we are
willing to assume that $W_n \asymp n/s$ in the DFS-induced chain, the
DFS fused lasso estimator is a $\log\log{n}$ factor away from the
optimal rate, provided that $s$ is not too large, namely
\smash{$s=O((\log{n} \log\log{n})^2)$}.  The spanning tree wavelet
estimator, on the other hand, is a \smash{$\log{n}$} factor from
optimal, without any real restrictions on $s$, i.e., it suffices to
have \smash{$s=O(n^a)$} for some $a>0$.  It is worth remarking that,
for large enough $s$, the lower bound in \eqref{eq:bd_minimax} is
perhaps not very interesting, as in such a case, we may as well
consider the bounded variation lower bound in \eqref{eq:bv_minimax},
which will likely be tighter (faster).

\section{Experiments}
\label{sec:experiments}

In this section we compare experimentally the speed and accuracy of
two approaches for denoising signals on graphs: the graph fused lasso,
and the fused lasso along the chain graph induced by a DFS ordering.
In our experiments, we see that the DFS-based denoiser sacrifices a
modest amount in terms of mean squared error, while providing gains 
(sometimes considerable) in computational speed. This shows that our
main theorem, in addition to providing new insights on MSE rates for
the graph fused lasso, also  
has important practice consequences.  For truly massive problems, 
where the full graph denoising problem is impractical to solve, we may 
use the linear-time DFS fused lasso denoiser, and obtain a favorable 
tradeoff of accuracy for speed. 
 
\subsection{Generic graphs}

We begin by considering three examples of large graphs of (more or
less) generic structure, derived from road networks in three states:  
California, Pennsylvania, and Texas. Data on these road networks are 
freely available at   
\url{https://snap.stanford.edu}.  In these networks, intersections   
and endpoints are represented by nodes, and roads connecting these
intersections or endpoints are represented by undirected edges; see
\cite{leskovec2009community} for more details. For each network, we
use the biggest connected component as our graph 
structure to run comparisons. The graph corresponding to 
California has $n=1957027$ nodes and $m=2760388$ edges, the one for
Pennsylvania has $n=1088092$ nodes and $m=1541898$ edges, and the
graph for Texas has $n=1351137$ nodes 
and $m=1879201$ edges. We compare Laplacian smoothing versus
the fused lasso over a DFS-induced chain, on the graphs from the three
states. We do not compare with the fused lasso over the original
graphs, due to its prohibitive computational cost at such large
scales. 

We used the following procedure to construct a synthetic signal
$\theta_0 \in \R^n$ on each of the road network graphs, of piecewise
constant nature:  
\begin{itemize}
\item an initial seed node $v_1$ is selected uniformly at random from
  the nodes $V=\{1,\ldots,n\}$ in the graph;    
\item a component $C_1$ is formed based on the 
  \smash{$\lfloor n/10 \rfloor$} 
  nodes closest to $v_1$ (where the distance between two
  nodes in the graph is given by the length of the shortest path 
  between them);  
\item a second seed node $v_2$ is selected uniformly at random from
  $G\setminus C_1$;  
\item a component $C_2$ is formed based on the \smash{$\lfloor n/10
    \rfloor$} nodes closest to $v_2$ (again in shortest path
  distance);
\item this process is repeated until we have a partition
  $C_1,\ldots,C_{10}$ of the node set $V$ into components of (roughly)
  equal size, and $\theta_0 \in \R^n$ is defined to take 
  constant values on each of these components. 
\end{itemize}
In our experiments, we considered 20 values of the total variation for
the underlying signal.  For each, the signal $\theta_0$ was scaled 
appropriately to achieve the given total variation value, and data $y
\in \R^n$
was generated by adding i.i.d.\ $N(0,0.2^2)$ noise to the
components of $\theta_0$.  For each data instance $y$, the DFS fused
lasso and Laplacian smoothing estimators, the former defined by
\eqref{eq:fused_lasso_dfs} and the latter by  
\begin{equation}
\label{eq:lap_smooth}
\htheta_{\mathrm{Lap}} \;=\; \argmin_{\theta \in \R^n} \;  
\half \|y - \theta\|_2^2 + \lambda \theta^\top L_G \theta,
\end{equation}
where \smash{$L_G = \nabla_G^\top \nabla_G$} is the Laplacian 
matrix of the given graph $G$, and each estimator is computed over 20 
values of its own tuning parameter.  Then, the value of the tuning
parameter minimizing the average MSE, over 50 draws of data $y$
around $\theta_0$, was selected for each method.  Finally, this
optimized MSE, averaged over the 50 draws of data $y$, and further,
over 10 repetitions of the procedure for constructing the signal
$\theta_0$ explained above, was recorded.  Figure
\ref{fig:state_graphs} displays the optimized MSE for the DFS  
fused lasso and Laplacian smoothing, as the total variation of the
underlying signal varies, for the three road network graphs. 

\begin{figure}[htb]
\centering
\includegraphics[width=0.325\textwidth]{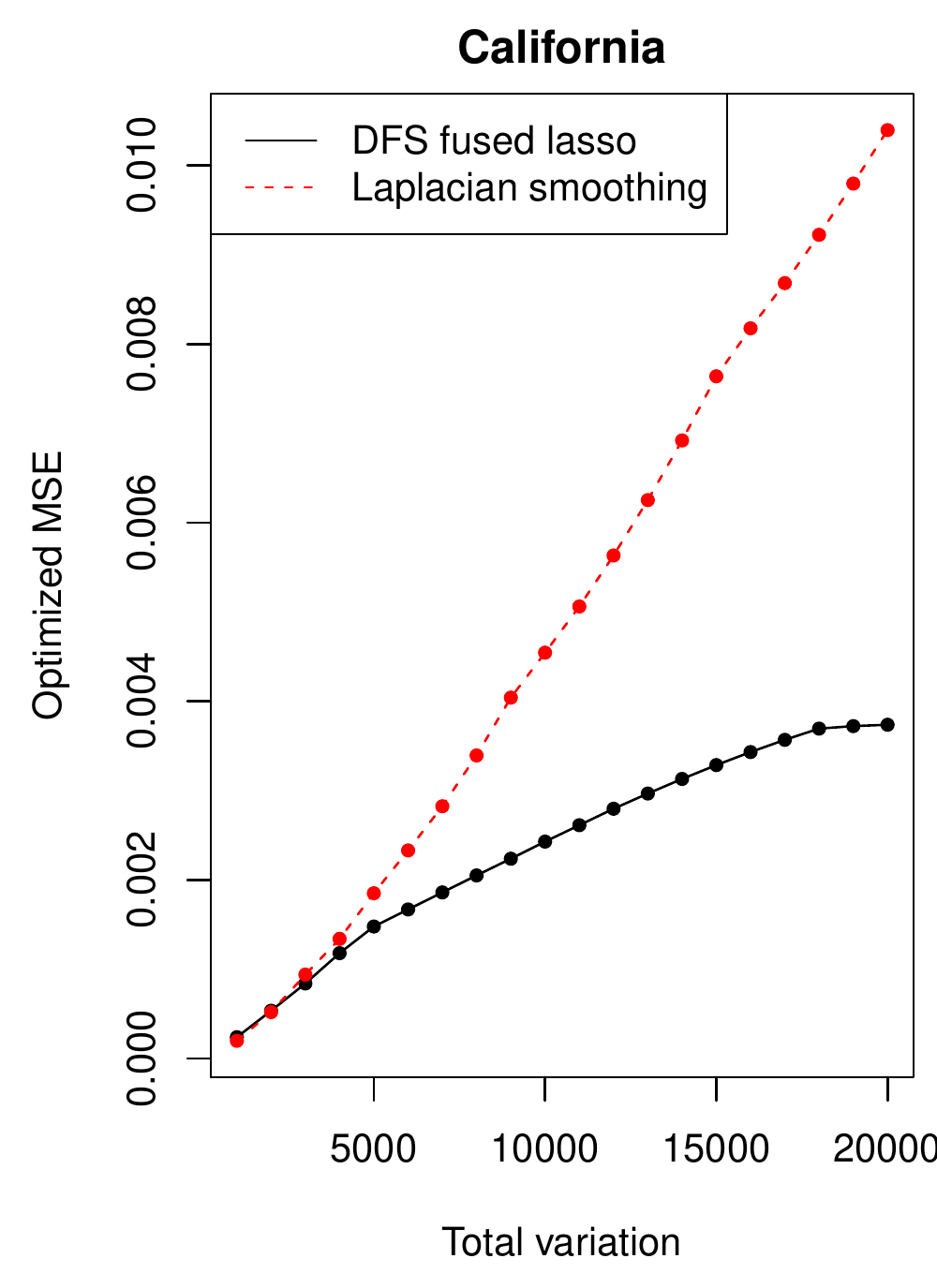} 
\includegraphics[width=0.325\textwidth]{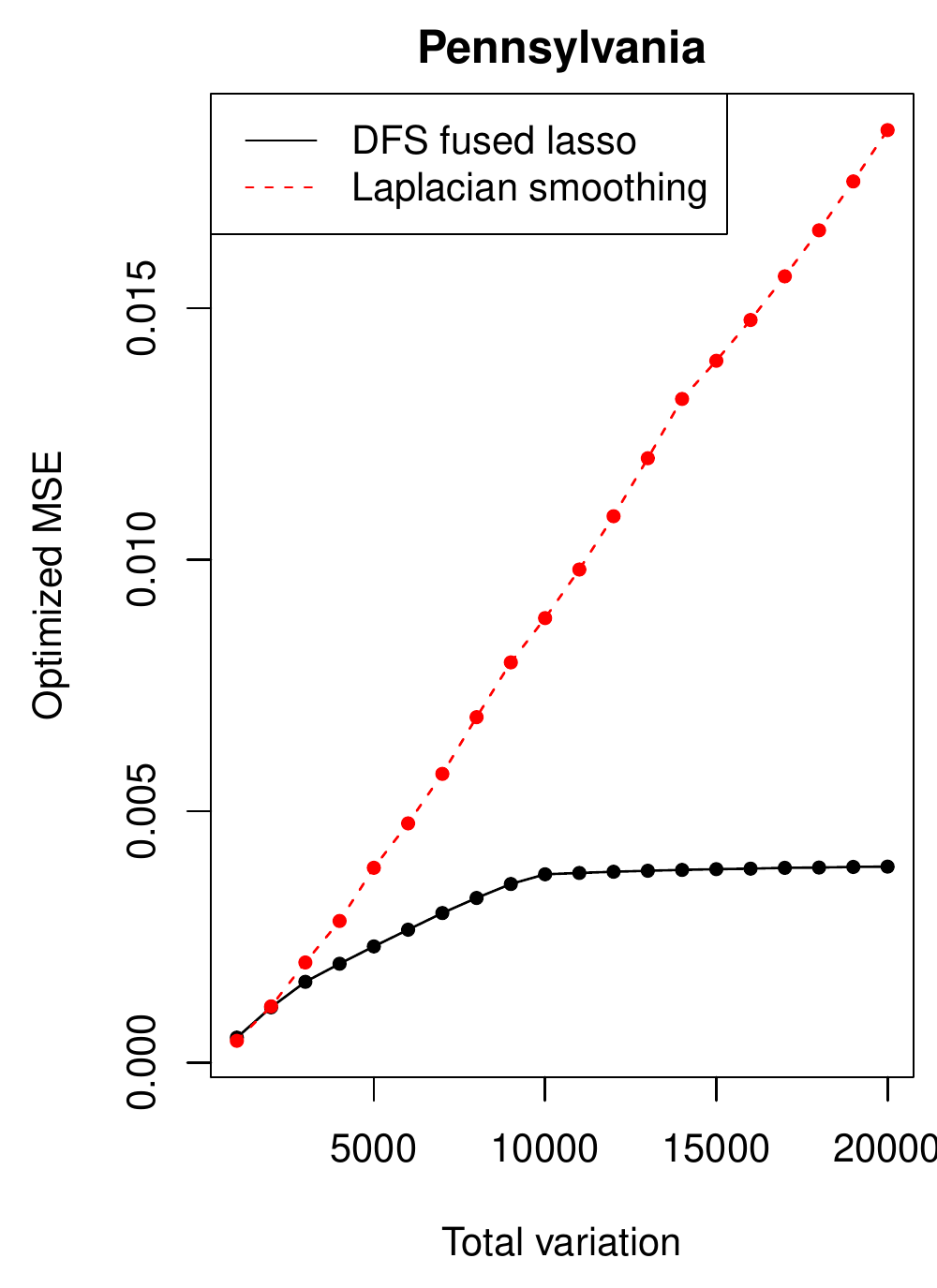} 
\includegraphics[width=0.325\textwidth]{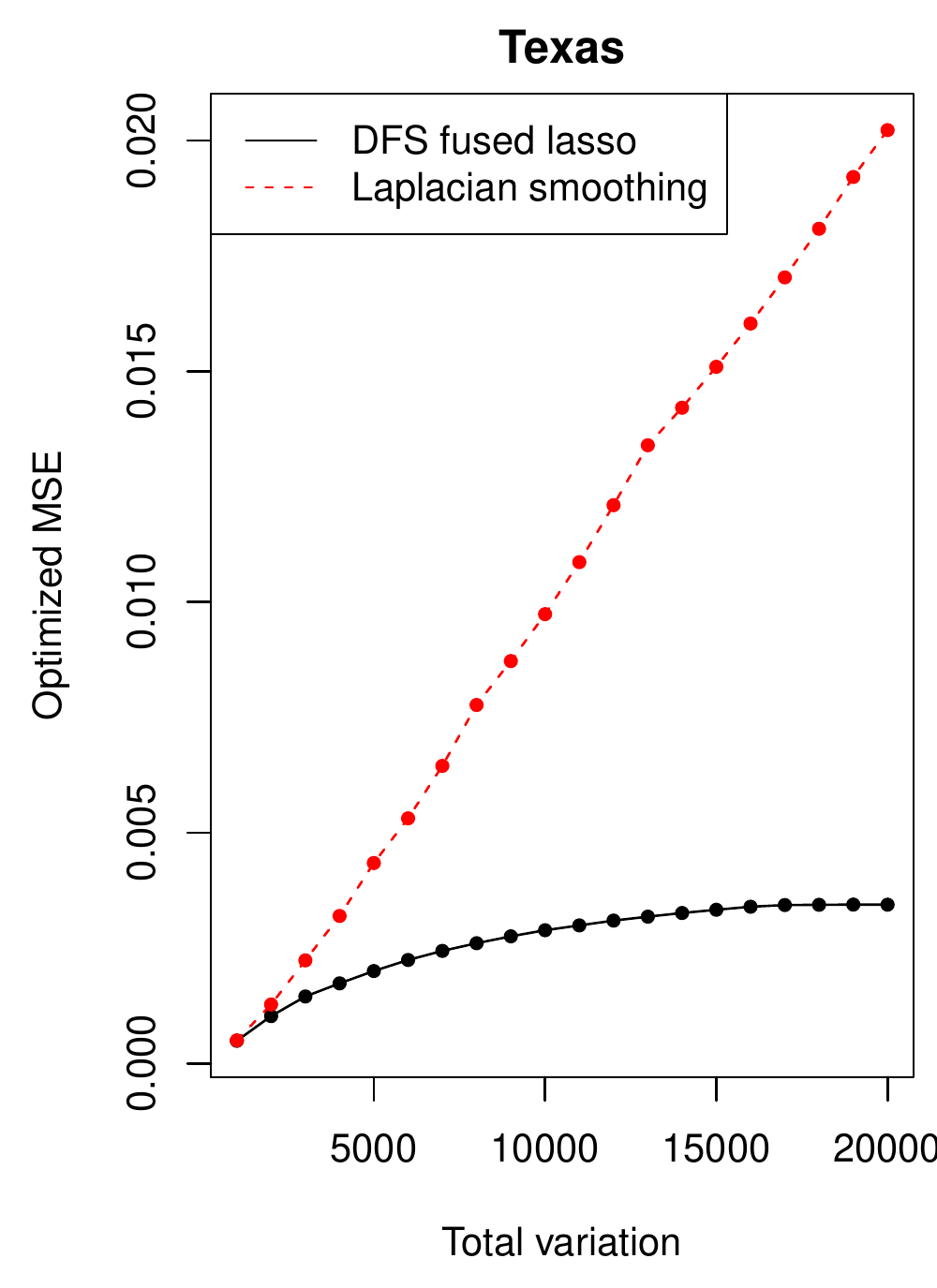} 
\caption{\small\it  The optimized MSE for the DFS fused lasso and 
  Laplacian smoothing (i.e., MSE achieved by these methods under
  optimal tuning) is plotted as a function of the total variation
  of the underlying signal, for each of the three road network
  graphs. This has been averaged over 50 draws of data $y$ for each 
  construction of the underlying signal $\theta_0$, and 10 repetitions 
  in constructing $\theta_0$ itself. For low values of
  the underlying total variation, i.e., low SNR levels, the two methods
  perform about the same, but as the SNR increases, the DFS fused
  lasso outperforms Laplacian smoothing by a considerable margin.}  
\label{fig:state_graphs}
\end{figure}

As we can see from the figure, for low values of the underlying total
variation, i.e., low signal-to-noise ratio (SNR) levels, Laplacian
smoothing and the DFS fused lasso, each tuned to optimality, perform
about the same.  This is because at low enough SNR levels, each will
be approximating $\theta_0$ by something like \smash{$\bar{y} \one$},
with \smash{$\bar{y}$} being the sample average of the data vector
$y$.  But as the SNR increases, we see that the DFS fused lasso
outpeforms Laplacian smoothing by a considerable amount. 
This might seem surprising, as Laplacian smoothing
uses information from the entire graph, whereas the DFS
fused lasso reduces the rich structure of the road network
graph in each case to that of an embedded chain.   
However, Laplacian smoothing is a linear smoother (meaning that
\smash{$\htheta_{\mathrm{Lap}}$} in \eqref{eq:lap_smooth} is a linear
function of the data $y$), and therefore it comes with
certain limitations when estimating signals of bounded variation
(e.g., see the seminal work of \citet{donoho1998minimax}, and the
more recent graph-based work of \citet{sadhanala2016total}).  In
contrast, the DFS fused lasso is a nonlinear estimator, and while it 
discards some information in the original graph structure, it
retains enough of the strong adaptivity properties of the fused
lasso over the original graph to statistically dominate a linear
estimator like Laplacian smoothing. 

Lastly, in terms of computational time, it took an average of 82.67
seconds, 44.02 seconds, and 54.49 seconds to compute the 20 DFS fused
lasso solutions (i.e., over the 20 tuning parameter values) for the
road network graphs from California, Pennsylvania, and Texas,
respectively (the averages are taken over the 50 draws of data $y$
around each signal $\theta_0$, and the 10 repetitions in constructing
$\theta_0$).  By comparison, it took an average of 2748.26 seconds,
1891.97 seconds, and 1487.36 seconds to compute the 20 Laplacian
smoothing solutions for the same graphs.  The computations and timings
were performed on a standard laptop computer (with a 2.80GHz Intel
Core i7-2640M processor). For the DFS fused lasso, in each problem 
instance, we first computed a DFS ordering using the {\tt dfs}
function from the R package {\tt igraph}, which is an R wrapper for a
C++ implementation of DFS, and initialized the algorithm at a random
node for the root. We then computed the appropriate 1d 
fused lasso solutions using the {\tt trendfilter} function from the R
package {\tt glmgen}, which is an R wrapper for a C++ implementation
of the fast (linear-time) dynamic programming algorithm in
\citet{johnson2013dynamic}. For Laplacian smoothing, we
used the {\tt solve} function from the R package {\tt Matrix}, which
is an R wrapper for a C++ implementation of the sparse Cholesky-based 
solver in \citet{davis2009dynamic}.  For such large graphs,
alternative algorithms, such as (preconditioned) conjugate gradient 
methods, could certainly be more efficient in computing Laplacian
smoothing solutions; our reported timings are only meant to indicate
that the DFS fused lasso is efficiently computable at problem sizes
that are large enough that even a simple linear method like Laplacian
smoothing becomes nontrivial. 

\subsection{2d grid graphs}

Next we consider a denoising example on a 2d grid graph of dimension
$1000 \times 1000$, so that the number of nodes is $n=1000000$ and the
number of edges is $m=1998000$.  We generated a synthetic piecewise
constant signal $\theta_0 \in \R^{1000 \times 1000}$ over the 2d grid,
shown in the top left corner of Figure \ref{fig:grid_graph_image},
where a color scale (displayed in the accompanying color legend) is
used, with red denoting the smallest possible value and yellow the
largest possible value. Data $y \in \R^{1000 \times 1000}$ was
generated by adding i.i.d.\ $N(0,1)$ noise to the components of
$\theta_0$, displayed in the top middle panel of Figure
\ref{fig:grid_graph_image}.  We then computed the 2d fused lasso
solution (i.e., the fused lasso solution over the full 2d grid graph),
as well as three DFS-based variations: the DFS fused lasso solution
using a random DFS ordering (given by running DFS beginning at a
random node), labeled as ``1 random DFS'' in the figure; the average
of DFS fused lasso solutions over 5 random DFS orderings, labeled ``5
random DFS'' in the figure; and the average of DFS fused lasso
solutions over 2 ``snake'' DFS orderings (one given by collecting and
joining all horizontal edges and the other all vertical edges) labeled
``2 snake DFS'' in the figure.  The tuning parameter for each method
displayed in the figure was chosen to minimize the average MSE over
100 draws of the data $y$ from the specified model.  Visually, we can
see that the full 2d fused lasso solution is the most accurate,
however, the 1 random DFS, 5 random DFS, and 2 snake DFS solutions all
still clearly capture the structure inherent in the underlying signal.
Of the three DFS variations, the 5 random DFS estimator is visually
most accurate; the 1 random DFS estimator is comparably ``blotchy'',
and the 2 snake DFS estimator is comparably ``stripey''.  

The left panel of \ref{fig:grid_graph_mse} shows the optimized MSE
for each method, i.e., the minimum of the average MSE over 100 draws
of the data $y$, when we consider 20 choices for the tuning parameter.    
This optimized MSE is plotted as a function of the sample size,
which runs from $n=2500$ (a $50\times 50$ grid) to $n=1000000$
(a $1000\times 1000$ grid), and in each case the underlying signal is
formed by taking an appropriate (sub)resolution of the image in the
top left panel of Figure \ref{fig:grid_graph_image}. 
  The 2d fused lasso provides the fastest
decrease in MSE as $n$ grows, followed by the 5 random DFS estimator,
then the 1 random DFS estimator, and the 2 snake DFS estimator.  This
is not a surprise, since the 2d fused lasso uses the information from 
the full 2d grid.  Indeed,
comparing \eqref{eq:grid_bound} and \eqref{eq:bv_bound_dfs}, we recall
that the 2d fused lasso enjoys an MSE rate of \smash{$t
  \log^2{n} /n$} when $\theta_0$ has 2d total variation $t$, whereas
the DFS fused lasso has an MSE rate of only
\smash{$(t/n)^{2/3}$} in this 
setting. When \smash{$t \asymp \sqrt{n}$}, which is a natural scaling
for the underlying total variation in 2d and also the scaling
considered in the experimental setup for the figure, these rates are
\smash{$(\log^2{n}) n^{-1/2}$} for the 2d fused lasso, and
\smash{$n^{-1/3}$} for the DFS fused lasso.  The figure uses a log-log
plot, so the MSE curves all appear
to have linear trends, and the fitted slopes roughly match these
theoretical MSE rates (-0.58 for the 2d fused lasso, and -0.39, -0.40,
and -0.36 for the three DFS variations). 

The right panel of Figure \ref{fig:grid_graph_mse} shows the runtimes
for each method (averaged over 100 draws of the data $y$), as a
function of the sample size $n$.  The runtime for each method counts
the total time taken to compute solutions across 20 tuning parameter
values.  The computations and timings were carried out on a standard
desktop computer (with a 3.40GHz Intel Core i7-4770 processor).  To
compute 2d fused lasso solutions, we used the {\tt TVgen} function in
the Matlab package {\tt proxTV}, which is a Matlab wrapper for a C++
implementation of the proximal stacking technique described in
\citet{barbero2014modular}.  For the DFS fused lasso, we computed
initial DFS orderings using the {\tt dfs} 
function from the Matlab package {\tt MathBGL}, and then, as before,
used the C++ implementation available through {\tt glmgen} to compute
the appropriate 1d fused lasso solutions.  The figure uses a log-log
plot, and hence we can see that all DFS-based estimators are quite a 
bit more efficient than the 2d fused lasso estimator.

\begin{figure}[htbp]
\centering
\begin{tabular}{cccc}
\includegraphics[height=0.28\textwidth]{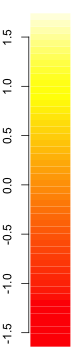} & 
\includegraphics[width=0.28\textwidth]{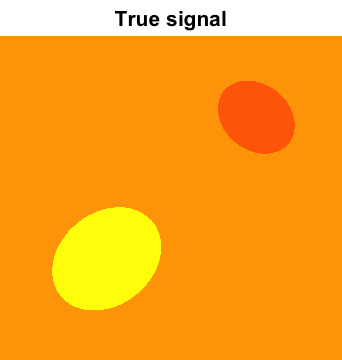} &
\includegraphics[width=0.28\textwidth]{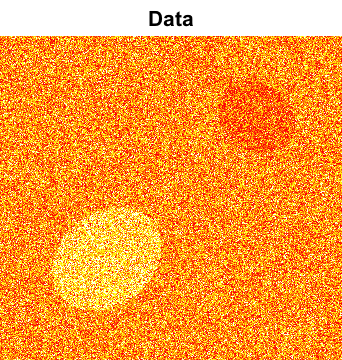} &
\includegraphics[width=0.28\textwidth]{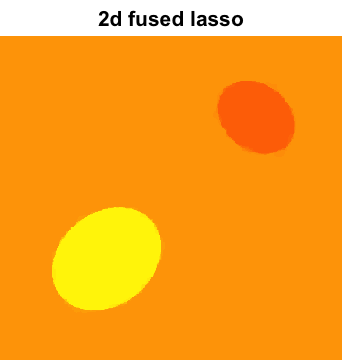} \\ 
& \includegraphics[width=0.28\textwidth]{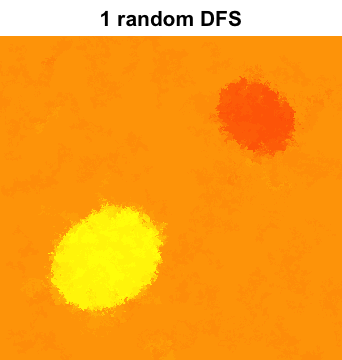} & 
\includegraphics[width=0.28\textwidth]{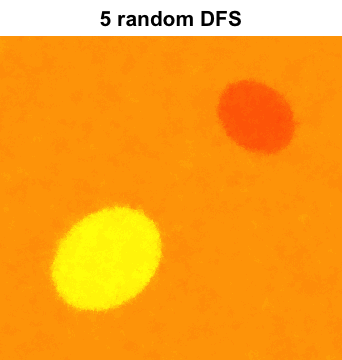} &
\includegraphics[width=0.28\textwidth]{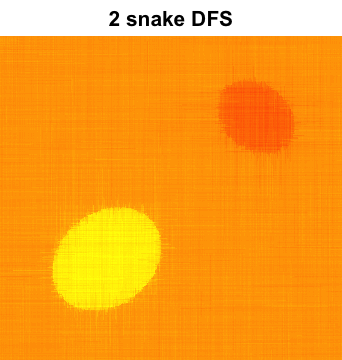} 
\end{tabular}
\caption{\small\it Underlying signal, data, and solutions from the
  2d fused lasso and different variations on the DFS fused lasso fit
  over a $1000 \times 1000$ grid.}
\label{fig:grid_graph_image}

\bigskip
\includegraphics[width=0.48\textwidth]{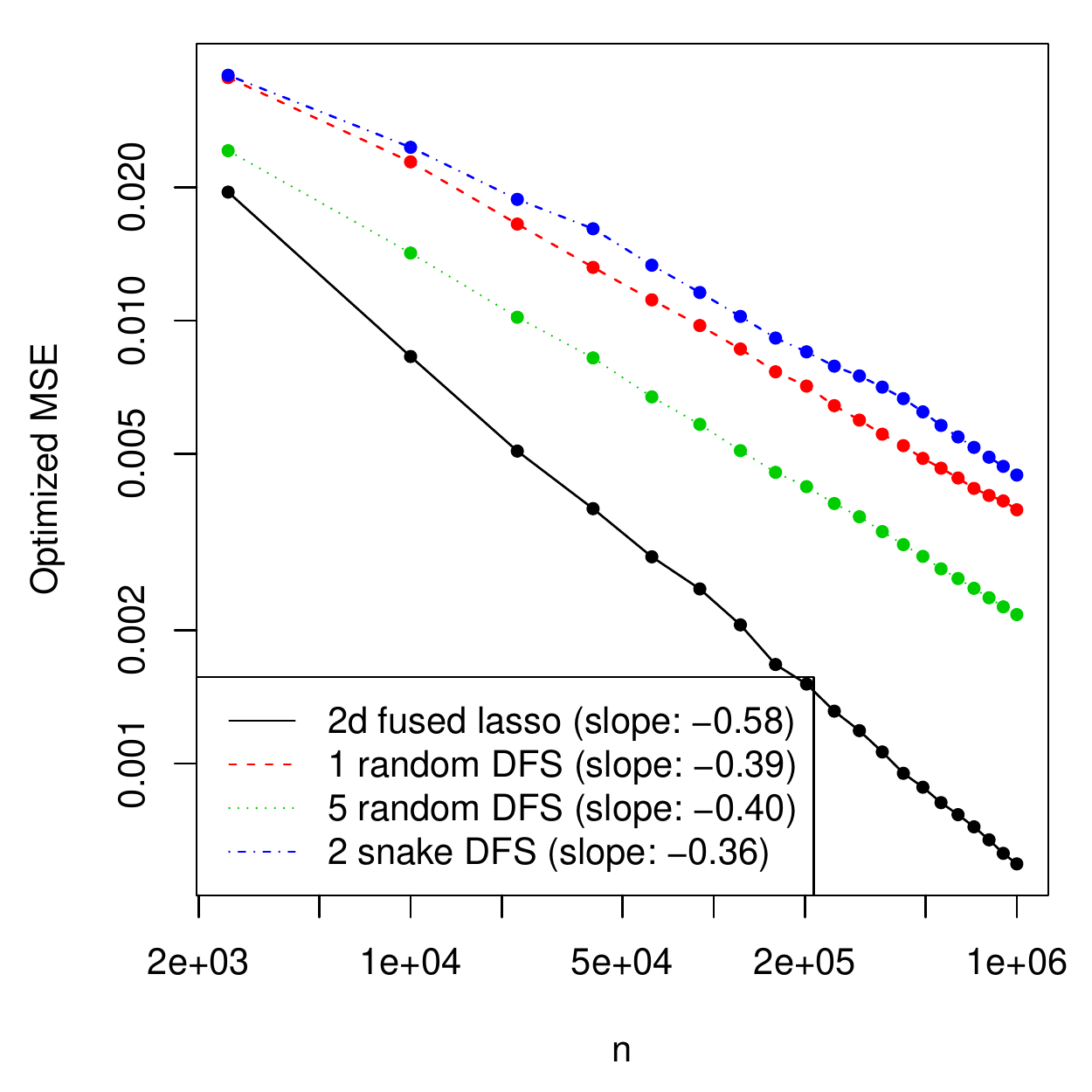}
\includegraphics[width=0.48\textwidth]{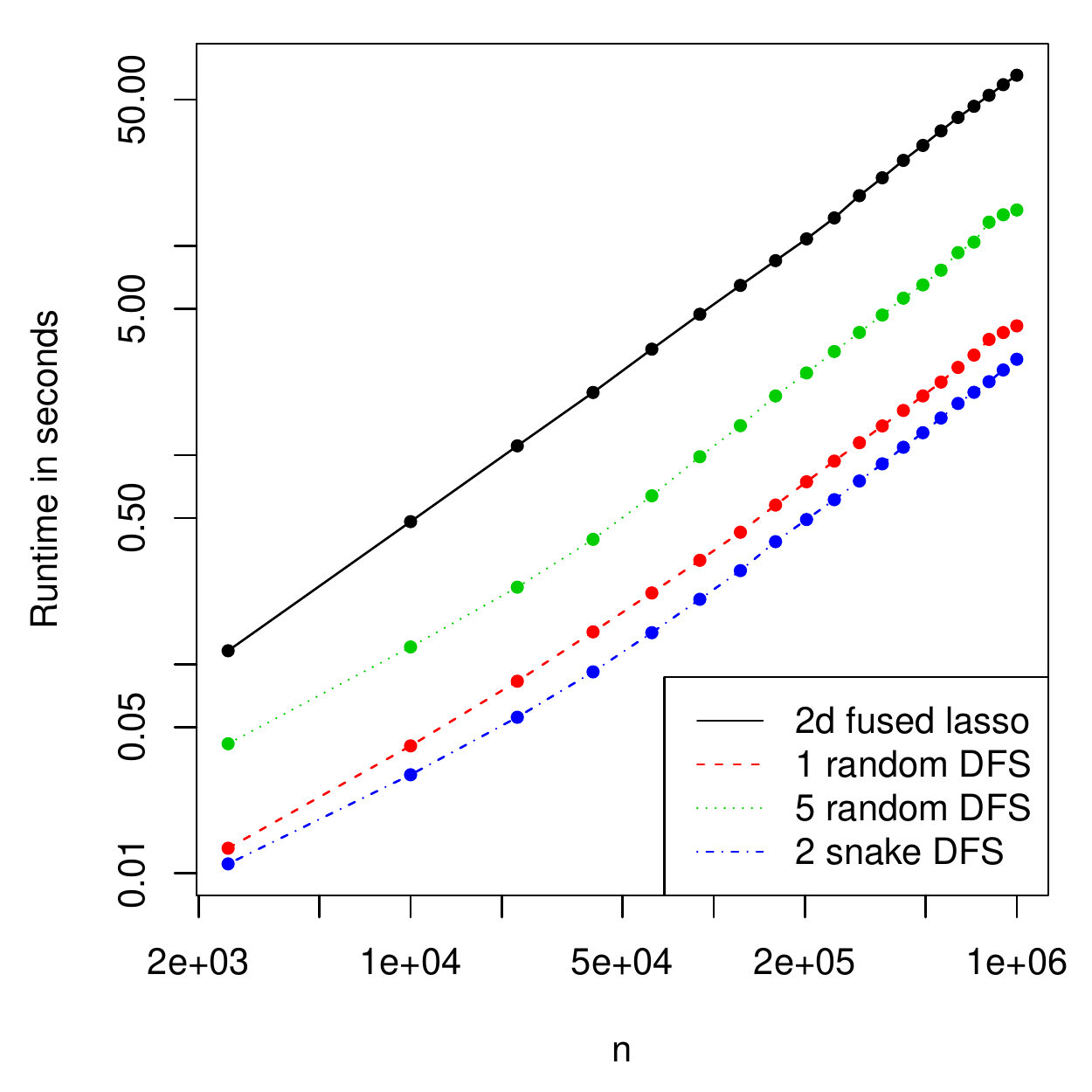}
\caption{\small\it Optimized MSE and runtime for the 2d fused lasso
  and DFS fused lasso estimators over a 2d grid, as the grid
  size $n$ (total number of nodes) varies.}
\label{fig:grid_graph_mse}
\end{figure}

\subsection{Tree graphs}
 
We finish with denoising comparisons on tree graphs, for sample sizes
varying from $n=100$ to $n=5300$.  For each sample size $n$, a random
tree is constructed via a sequential process in which each node is
assigned a number of children between 2 and 10 (uniformly at random).
Given a tree, an underlying signal $\theta_0 \in \R^n$ is constructed
to be piecewise constant with total variation \smash{$5\sqrt{n}$} (the
piecewise constant construction here is made easy because the oriented
incidence matrix of a tree is invertible).  Data $y \in \R^n$ was
generated by adding i.i.d.\ $N(0,1)$ noise to $\theta_0$. We
compared the fused lasso estimator over the full tree, 1 random DFS
and 5 random DFS estimators (using the terminology from the last
subsection), and the wavelet smoothing estimator defined in
\eqref{eq:wavelet}.
For each estimator, we computed the entire solution path using the
path algorithm of \citet{tibshirani2011solution} implemented in  
the R package {\tt genlasso}, and selected the step along the
path to minimize the average MSE over 50 draws of data $y$ around
$\theta_0$, and 10 repetitions in constructing $\theta_0$. (The full
solution path can be computed here because each estimator can be cast
as a generalized lasso problem, and because the problem sizes
considered here are not enormous.)

 The left panel of Figure \ref{fig:tree_graph} 
plots this optimized MSE as a function of the sample size $n$.  
We see that the fused lasso estimator over the full tree and the 5
random DFS estimator perform more or less equivalently over all
sample sizes.  The 1 random DFS estimator is slightly worse, and the
wavelet smoothing estimator is considerably worse. 
The right panel shows the the MSE as a function of the effective 
degrees of freedom of each estimator, for a particular data instance
with $n=5300$.  We see that both the tree fused lasso and 1 random DFS
estimators achieve their optimum MSEs at solutions of low complexity  
(degrees of freedom), whereas wavelet smoothing does not come  
close to achieving this MSE across its entire path of solutions. 

\begin{figure}[htbp]
\centering
\includegraphics[width=0.48\textwidth]{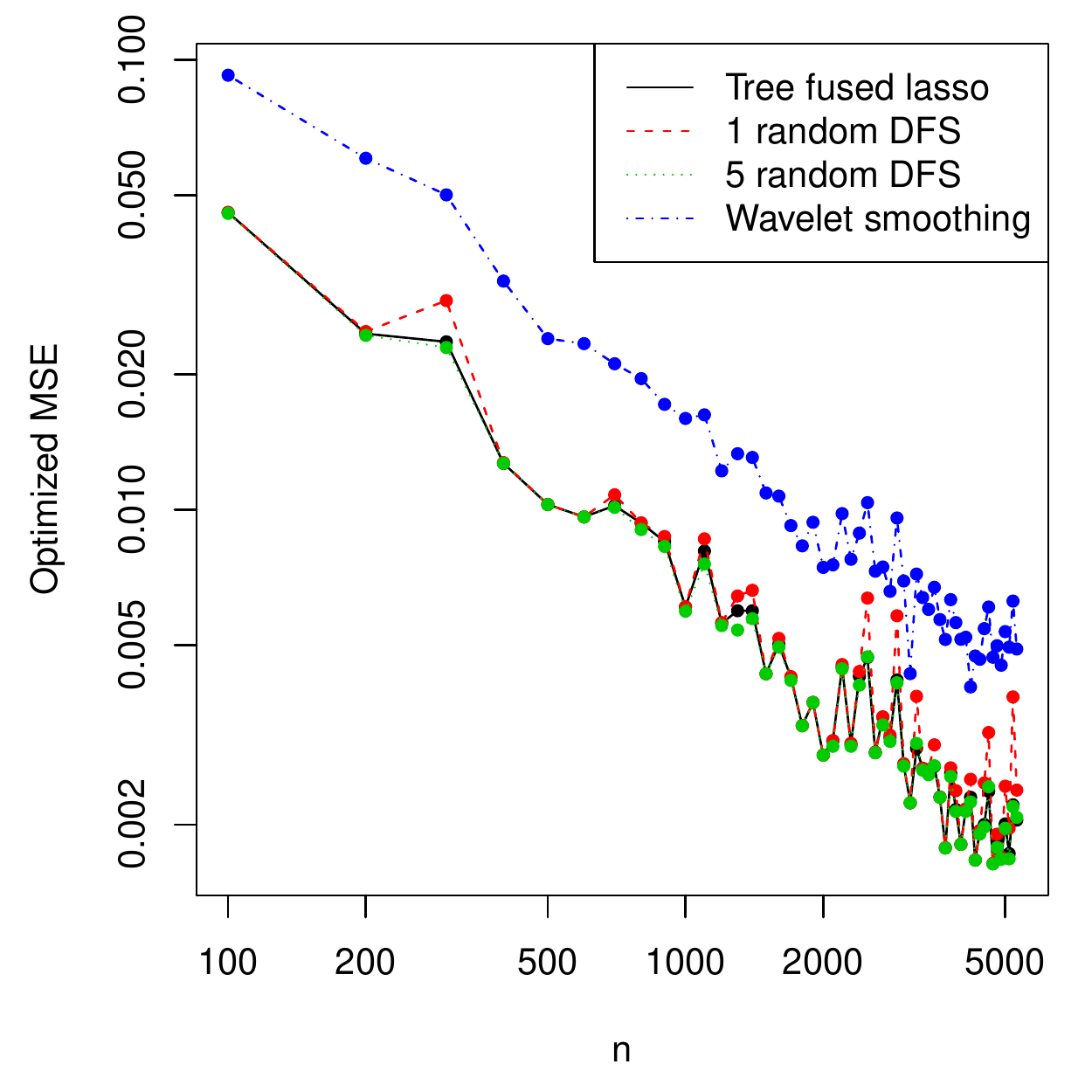}
\includegraphics[width=0.48\textwidth]{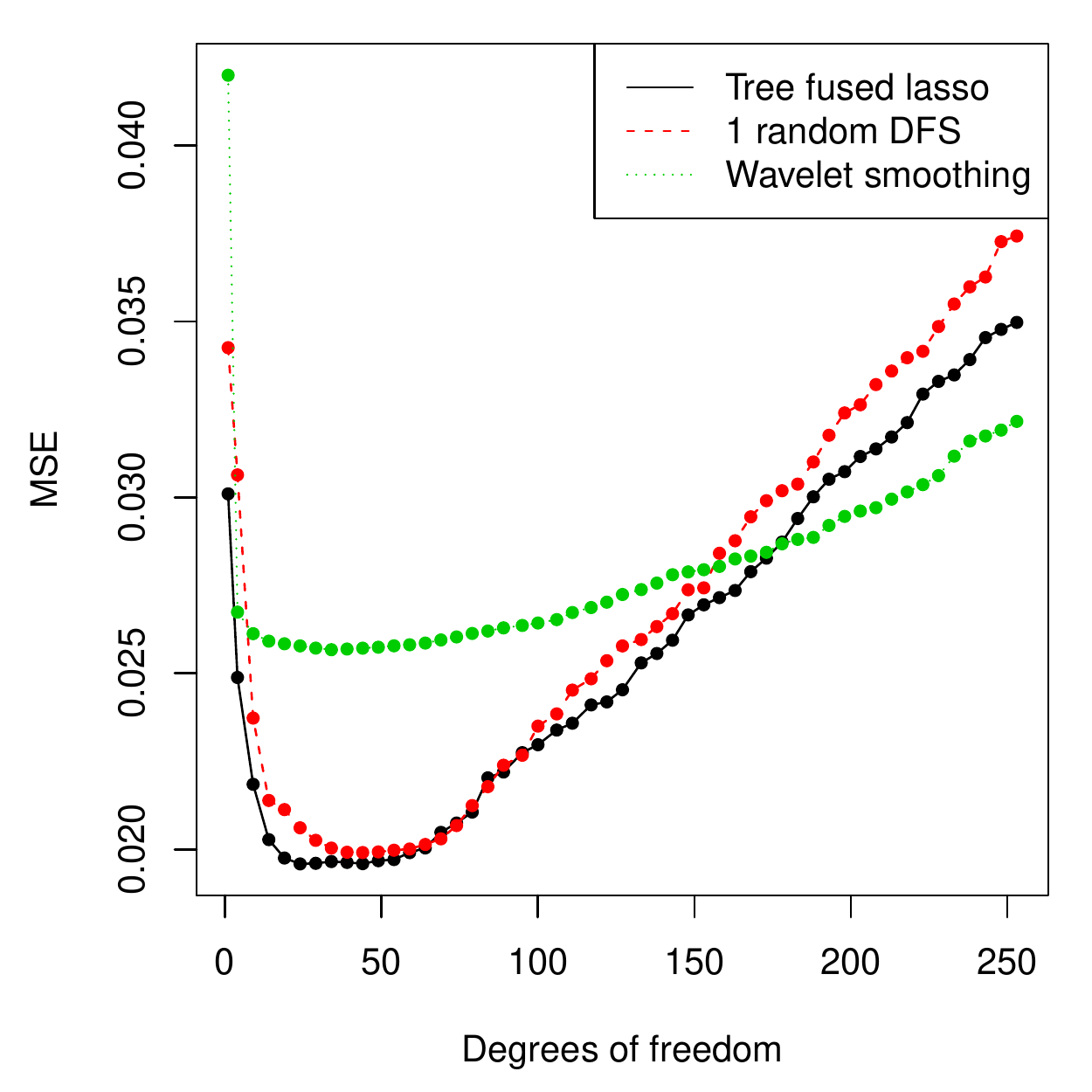}
\caption{\small\it The left panel shows the optimized MSE as a
  function of the sample size for the fused lasso over a tree graph,
  as well as the 1 random DFS and 5 random DFS estimators, and wavelet
  smoothing. The right panel}
\label{fig:tree_graph}
\end{figure}

\section{Discussion}
\label{sec:discussion}

Recently, there has been a significant amount on interest on  
graph-structured denoising. Much of this work has focused on the
construction of graph kernels or wavelet bases.  We have proposed and
studied a simple method, defined by computing the 1d fused lasso over
a particular DFS-induced ordering of the nodes of a general
graph. This linear-time algorithm comes with strong theoretical
guarantees for signals of bounded variation (achieving optimal MSE
rates for trees of bounded degree), as well as guarantees for
signals with a bounded number of nonzero differences (achieving nearly
optimal rates under a condition on the spacings of jumps along the
DFS-induced chain).  We summarize our theoretical results in Table
\ref{tab:summary}.

\begin{table}
\centering
\bgroup
\def\arraystretch{1.75}
\begin{tabular}{|c|c|c|}
\hline
& \smash{$\BV_G(t)$}, $t \asymp 1$
& \smash{$\BD_G(s)$}, $s \asymp 1$ \\    
\hline 
Fused lasso, \smash{$\htheta_G$} 
& \smash{$n^{-2/3}$} 
& unknown \\ 
\hline 
Spanning tree wavelets, \smash{$\htheta_W$} 
& unknown 
& \smash{$(\log^2{n} \log d_{\mathrm{max}})/n$} \\  
\hline
DFS fused lasso, \smash{$\htheta_{\mathrm{DFS}}$} 
& \smash{$n^{-2/3}$}
& \smash{$(\log{n}\log\log{n})/n^*$} \\ 
\hline
Tree lower bound 
& \smash{$n^{-2/3} d_{\mathrm{max}}^{-4/3}$} 
& \smash{$\log{n}/n$} \\ \hline 
\end{tabular}
\egroup
\caption{\small\it A summary of the theoretical results derived in
  this paper.  All rates are on the mean squared error (MSE) scale
  (\smash{$\E\|\htheta - \theta_0 \|_n^2$} for an estimator 
  \smash{$\htheta$}), and for simplicity, are presented under a
  constant scaling for $t,s$, the radii in the \smash{$\BV_G(t),
    \BD_G(s)$} classes, respectively.  The superscript ``$^*$'' in the 
  \smash{$\BD_G(s)$} rate for the DFS fused lasso is used to emphasize 
  that this rate only holds under the assumption that $W_n \asymp
  n$. Also, we write \smash{$d_{\mathrm{max}}$} to denote the max
  degree of the graph in question.}
\label{tab:summary}
\end{table}

Practically, we have seen that the DFS fused lasso can often represent 
a useful tradeoff between computational efficiency and statistical
accuracy, versus competing methods that offer better statistical
denoising power but are more computationally expensive, especially for
large problems.  A simple trick like averaging multiple DFS fused
lasso fits, over multiple random DFS-induced chains, often improves
statistical accuracy at little increased computational cost.  Several
extensions along these lines, and other lines, are possible.  To
study any of them in detail is beyond the scope of this paper.  We
discuss them briefly below, leaving detailed follow-up to future
work.

\subsection{Beyond simple averaging}

Given multiple DFS fused lasso estimators, 
\smash{$\htheta^{(1)}_{\mathrm{DFS}}, \ldots,
  \htheta^{(K)}_{\mathrm{DFS}}$}, obtained using multiple DFS-induced 
chains computed on the same graph $G$, there are several 
possibilities for intelligently combining these estimators beyond the
simple average, denoted (say) \smash{$\bar\theta^{(K)}_{\mathrm{DFS}}
  = (1/K) \sum_{k=1}^K \htheta^{(k)}_{\mathrm{DFS}}$}.  To better
preserve edges in the combined estimator, we could run a simple
nonlinear filter---for example, a median filter, over
\smash{$\htheta^{(1)}_{\mathrm{DFS}}, \ldots,  
  \htheta^{(K)}_{\mathrm{DFS}}$} (meaning that the combined estimator
is defined by taking medians over local neighborhoods of all of the 
individual estimators). A more sophisticated
approach would be to compute the DFS fused lasso estimators    
sequentially, using the $(k-1)$st estimator to modify the response in
some way in the 1d fused lasso problem that defines the $k$th 
DFS fused lasso estimator. We are intentionally vague here with the 
specifics, because such a modification could be implemented in
various ways; for example, it could be useful to borrow ideas from the
boosting literature, 
which would have us treat each DFS fused lasso estimator as a weak 
learner.   


\subsection{Distributed algorithm}
 
For large graphs, we should be able to both compute a DFS 
ordering over $G$, and solve the DFS fused lasso problem in
\eqref{eq:fused_lasso_dfs}, in a distributed fashion.  There are
many algorithms for distributed DFS, offering a variety of 
communication and time complexities; see, e.g.,
\citet{tsin2002remarks} for a survey.  Distributed algorithms for the 
1d fused lasso are not as common, though we can appeal to the now 
well-studied framework for distributed optimization via the
alternating direction method of multipliers (ADMM) from   
\citet{boyd2011distributed}.  Different formulations for
the auxiliary variables present us with different options for
communication costs.  We have found that, for a formulation that
requires $O(1)$-length messages to be communicated between processors, 
the algorithm typically converges in a reasonably small number of
iterations. 

\subsection{Theory for piecewise constant signals}

The bounded differences class \smash{$\BD_G(s)$} in
\eqref{eq:bd_class} is defined in terms of the cut metric
\smash{$\|\nabla_G \theta\|_0$} of a parameter $\theta$, which recall,
counts the number of nonzero differences occurring in $\theta$ over
edges in the graph $G$. The cut metric measures a notion of strong 
sparsity (compared to the weaker notion measured by the total
variation metric) in a signal $\theta$, over edge differences;
but, it may not be measuring sparsity on 
the ``right'' scale for certain graphs $G$. Specifically, the
cut metric \smash{$\|\nabla_G \theta\|_0$} can actually be quite
large for a parameter $\theta$ that is piecewise constant over $G$,
with a small number of pieces---these are groups of connected nodes
that are assigned the same constant value in $\theta$. 
Over the 2d grid graph, e.g., one can easily define a parameter
$\theta$ that has only (say) two constant pieces but on the order of   
\smash{$\sqrt{n}$} nonzero edge differences. Therefore, for such a 
``simple'' configuration of the parameter $\theta$, the cut metric 
\smash{$\|\nabla_G \theta\|_0$} is deceivingly large. 

To formally define a metric that measures the number of constant
pieces in a parameter $\theta$, with respect to a graph $G=(V,E)$, we  
introduce a bit of notation.  Denote by $Z(\theta) \subseteq E$ the
subset of edges over which $\theta$ exhibits differences
of zero, i.e., \smash{$Z(\theta) = \{ e \in E : \theta_{e^+} =
  \theta_{e^-} \}$}.  Also write
\smash{$(\nabla_G)_{Z(\theta)}$} for the submatrix of the edge
incidence matrix $\nabla_G$ with rows indexed by $Z(\theta)$. 
We define the {\it piece metric} by 
\begin{equation*}
\rho_G(\theta) = \nuli\big((\nabla_G)_{Z(\theta)}\big), 
\end{equation*}
where $\nuli(\cdot)$ denotes the dimension of the null space of its
argument.  An equivalent definition is
\begin{equation*}
\rho_G(\theta) = \text{the number of connected components in $(V,E
  \setminus Z(\theta))$}.
\end{equation*}
We may now define the {\it piecewise constant class}, with respect to
$G$, and a parameter $s>0$, 
\begin{equation*}
\PC_G(s) = \{\theta \in \R^n : \rho_G(\theta) \leq s\}.
\end{equation*}
It is not hard to to see that \smash{$\BV_G(s) \subseteq
  \PC_G(s)$} 
(assuming only that $G$ is connected), but for certain graph
topologies, the latter class \smash{$\PC_G(s)$} will be much larger.  
Indeed, to repeat what we have conveyed above, for the 2d grid one
can naturally define a parameter $\theta$ such that
\smash{$\theta \in \BD_G(\sqrt{n})$} and \smash{$\theta
  \in \PC_G(2)$}.

We conjecture that the fused lasso estimator over $G$ can achieve a 
fast MSE rate when the mean $\theta_0$ in \eqref{eq:model}
exhibits a small number of constant pieces, i.e., \smash{$\theta_0 \in
  \PC_G(s)$}, 
provided that these pieces are of roughly equal size.
Specifically, assuming \smash{$\rho_G(\theta_0) \leq s$}, let
$W_n$ denote the smallest size of a connected component in the graph  
\smash{$(V, E \setminus Z(\theta_0))$}.  Then, for a suitable choice
of $\lambda$, we conjecture that the fused lasso estimator
\smash{$\htheta_G$} in \eqref{eq:fused_lasso} satisfies
\begin{equation*}
\tag{{\bf conjecture}}
\|\htheta_G - \theta_0\|_n^2 = O_\P\bigg(\frac{s}{n} 
\Big(\mathrm{polylog}\,{n} + n/W_n \Big)\bigg),
\end{equation*}
where $\mathrm{polylog}\,{n}$ is a shorthand for a polynomial of  
$\log{n}$. This would substantially improve upon existing  
strong sparsity denoising results, such as \eqref{eq:grid_bound}, 
\eqref{eq:bd_bound_dfs}, \eqref{eq:bd_bound_wav}, since the latter
results are all proven for the class \smash{$\BD_G(s)$}, which, as we
have argued, can be much smaller than \smash{$\PC_G(s)$}, depending on
the structure of $G$. 

\subsection{Weighted graphs}
\label{sec:weighted}

The key result in Lemma \ref{lem:dfs} can be extended to the setting 
of a weighted graph $G=(V,E,w)$, with $w_e \geq 0$ denoting the edge 
weight associated an edge $e \in E$.  We state the following without
proof, since its proof follows in nearly the exact same way as that of
Lemma \ref{lem:dfs}.

\begin{lemma}
Let $G=(V,E,w)$ be a connected weighted graph, where recall we write
$V=\{1,\ldots,n\}$, and we assume all edge weights are nonnegative.
Consider running DFS on $G$, and denote by $\tau : \{1,\ldots,n\} \to 
\{1,\ldots,n\}$ the induced permutation, so that if $v_1,\ldots,v_n$
are the nodes in the order that they are traversed by DFS, then 
\begin{equation*}
\tau(i) = v_i, \quad \text{for all} \;\, i=1,\ldots,n.
\end{equation*}
Denote $w_{\mathrm{min}} = \min_{e \in E} w_e$, the minimum edge
weight present in the graph, and define
\begin{equation}
\label{eq:w_til}
\tilde{w}_{ij} = \begin{cases}
w_e & \text{if $e=\{i,j\} \in E$}, \\
w_{\mathrm{min}} & \text{otherwise},
\end{cases}
\quad \text{for all} \;\, i,j =1,\ldots,n.
\end{equation}
It holds that
\begin{equation}
\label{eq:dfs_weighted_l1}
\sum_{i=1}^{n-1} \tilde{w}_{\tau(i),\tau(i+1)}
\big| \theta_{\tau(i+1)} - \theta_{\tau(i)} \big| \leq 
2 \sum_{e \in E} w_e | \theta_{e^+} - \theta_{e^-} |,
\quad \text{for all} \;\, \theta \in \R^n,
\end{equation}
as well as 
\begin{equation}
\label{eq:dfs_weighted_l0}
\sum_{i=1}^{n-1} \tilde{w}_{\tau(i),\tau(i+1)}
1\big\{ \theta_{\tau(i+1)} \not= \theta_{\tau(i)} \big\} \leq  
2 \sum_{e \in E} w_e 1\{ \theta_{e^+} \not= \theta_{e^-} \}, 
\quad \text{for all} \;\, \theta \in \R^n.
\end{equation}
\end{lemma}

The bounds in \eqref{eq:dfs_weighted_l1}, \eqref{eq:dfs_weighted_l0}
are the analogies of \eqref{eq:dfs_l1}, \eqref{eq:dfs_l0} but for
a weighted graph $G$; indeed we see that we can still embed a DFS
chain into $G$, but this chain itself comes with edge weights, as in 
\eqref{eq:w_til}.  These new edge weights in the chain do not
cause any computational issues; the 1d fused lasso problem 
with arbitrary penalty weights can still be solved in $O(n)$ time
using the taut string algorithm in \citet{barbero2014modular}.  Thus,
in principle, all of the results in this paper should carry over in
some form to weighted graphs.

\subsection{Potts and energy minimization}

Replacing the total variation metric by the cut metric in the fused
lasso problem \eqref{eq:fused_lasso} gives us
\begin{equation}
\label{eq:potts}
\ttheta_G = \argmin_{\theta \in \R^n} \; 
\half \|y - \theta\|_2^2 + \lambda \| \nabla_G \theta\|_0,
 \end{equation}
often called the {\it Potts minimization} problem. 
Because the 1d Potts minimization problem
\begin{equation}
\label{eq:potts_1d}
\ttheta_{\mathrm{1d}} = \argmin_{\theta \in \R^n} \; 
\half \|y - \theta\|_2^2 + \lambda \| \nabla_{\mathrm{1d}} \theta\|_0  
 \end{equation}
can be solved efficiently, e.g., in worst-case $O(n^2)$ time with 
dynamic programming 
\citet{bellman1961approximation,johnson2013dynamic}, 
the same strategy that we have proposed in this paper can be applied
to reduce the graph Potts problem \eqref{eq:potts} to a 1d Potts
problem \eqref{eq:potts_1d}, via a DFS ordering of the nodes.  This
may be especially interesting as the original Potts problem 
\eqref{eq:potts} is non--convex and generally intractable (i.e.,
intractable to solve to global optimality) for an arbitrary graph
structure, so a reduction to a worst-case quadratic-time
denoiser is perhaps very valuable.

When the optimization domain in \eqref{eq:potts} is a discrete set,
the problem is often called an {\it energy minimization}
problem, as in \citet{boykov2001fast}.  It has not escaped our notice
that our technique of denoising over DFS-induced chains could be
useful for this setting, as well.

\appendix 

\section{Proofs}

\subsection{Derivation of \eqref{eq:gen_bound} from Theorem 3 in
  \citet{wang2016trend}}
\label{app:gen_bound}

We first establish a result on the exact form for the inverse of (an 
augmented version of) the edge incidence matrix of a generic tree
$T=(V,E_T)$, where, recall $V=\{1,\ldots,n\}$.  Without
a loss of generality, we may assume that the root of $T$ is at node
1. For $m\leq n$, we define a path in $T$, of length $m$, to be a 
sequence $p_1,\ldots, p_m$ such that $\{p_r,p_{r+1}\} \in E_T$ for
each $r = 1,\ldots, m-1$. We allow for the possibility that $m = 1$,
in which case the path has just one node. For any $j,k,\ell =
1,\ldots,n$, we say that $j$ is on the path from $k$ to $\ell$ if 
there exists a path $p_1,\ldots, p_m$  such that $p_1 = k$, $p_m = \ell$
and $p_r =  j$ for some $r = 1,\ldots,m$.  For each
node $i = 2,\ldots,n$ (each node other than the root), we 
define its parent $p(i)$ to be the node connected to $i$ which is on
the path from the root to $i$.   

We can also assume without a loss of generality that for each $i =
2,\ldots,n$, the $(i-1)$st row of $\nabla_T$ corresponds to the edge
$\{p(i), i\}$, and thus we can write  
\begin{equation*}
(\nabla_T)_{i-1,j} = \begin{cases}
-1 & \text{if $j=p(i)$}, \\
1 & \text{if $j=i$}, \\
0 & \text{if $j \in \{1,\ldots,n\} \setminus \{i, p(i)\}$}. 
\end{cases}
\end{equation*}
for each $j=1,\ldots,n$.
The next lemma describes the inverse of $\nabla_T$, in the appropriate
sense. 

\begin{lemma}
\label{lem:tree_edge_incidence_inverse}
Let $e_1 = (1,0,\ldots,0) \in \R^n$, and define the matrix $A_T \in 
\R^{n\times n}$ by
\begin{equation}
\label{eq:tree_edge_incidence_inverse}
(A_T)_{i,j} = \begin{cases}
1 & \text{if $j$ is on the path from the root to $i$}, \\
0 & \text{otherwise},
\end{cases}
\end{equation}
for each $i,j =1,\ldots,n$.
Then
\begin{equation*}
A_T  = \left( \begin{array}{c}
    e_1^\top  \\
    \nabla_T
  \end{array} \right)^{-1}.
\end{equation*}
\end{lemma}

\begin{proof}
We will prove that the product 
\[
B = \left( \begin{array}{c}
    e_1^\top  \\
    \nabla_T
  \end{array} \right) A_T
\]
is the identity.
As the root of $T$ corresponds to node 1, we have that by
definition of $A_T$ that its first column is
\[
(A_T)_{\cdot,1}\, = (1,\ldots,1),
\]
which implies that the first column of $B$ is
\[
B_{\cdot,1}= e_1. 
\]
Moreover, by definition of $A_T$, its first row is
\[
(A_T)_{1,\cdot} = e_1^\top,
\]
which implies that the first row of $B$ is 
\[
B_{1,\cdot} = e_1^\top.   
\]
Let us now assume that $i,j$ are each not the root. We proceed to 
consider three cases.

\smallskip\smallskip
\noindent
{\bf Case 1.} Let $j \neq i$, and $j$ be on the path from the root to
$i$. Then $j$ is also on the path from the root to $p(i)$. This
implies that 
\[
B_{ij} = \left( \begin{array}{c}
    e_1^\top  \\
    \nabla_T
  \end{array} \right)_{i,\cdot}
 (A_T)_{\cdot,j} = (\nabla_T)_{i-1,\cdot} (A_T)_{\cdot,j} = 1 - 1 = 0. 
\]

\smallskip\smallskip
\noindent
{\bf Case 2.} Let $j \neq i$, and $j$ not be on the path from the root
to $i$. Then $j$ is not on the path from the root to $p(i)$, which
implies that
\[
B_{ij} = \left( \begin{array}{c}
    e_1^\top  \\
    \nabla_T
  \end{array} \right)_{i,\cdot}
 (A_T)_{\cdot,j} = (\nabla_T)_{i-1,\cdot} (A_T)_{\cdot,j} = 0 - 0 = 0. 
\]

\smallskip\smallskip
\noindent
{\bf Case 3.} Let $j = i$. Then $j$ is on the path from the root to
$i$, and $j$ is not on the path from the root to $p(i)$. Hence, 
\[
B_{ij} = \left( \begin{array}{c}
    e_1^\top  \\
    \nabla_T
  \end{array} \right)_{i,\cdot}
 (A_T)_{\cdot,j} = (\nabla_T)_{i-1,\cdot} (A_T)_{\cdot,j} = -1 \cdot 0 
 + 1 \cdot 1 = 1. 
\]

\smallskip\smallskip
Assembling these three cases, we have shown that $B=I$, completing the
proof. 
\end{proof}

We now establish \eqref{eq:gen_bound}.

\begin{proof}[Proof of \eqref{eq:gen_bound}]
The proof of Theorem 3 in \citet{wang2016trend} proceeds as in
standard basic inequality arguments for the lasso, and arrives at the
step
\begin{equation*}
\|\Pi^\perp (\htheta_G - \theta_0)\|_2^2 \leq 2 \epsilon^\top
\Pi^\perp (\htheta_G - \theta_0) +
2\lambda \|\nabla_G \theta_0\|_1 - 2\lambda \|\nabla_G \htheta_G\|_1,
\end{equation*}
where $\Pi^\perp$ is the projection matrix onto the space
\smash{$\spa\{\one\}^\perp$}, i.e., the linear space of all vectors
orthogonal to the vector $\one=(1,\ldots,1) \in \R^n$ of all 1s. The
proof in \citet{wang2016trend} uses the identity
\smash{$\Pi^\perp=\nabla_G^\dagger \nabla_G$}, where 
\smash{$\nabla_G^\dagger$} denotes the pseudoinverse of $\nabla_G$.
However, notice that we may also write
\smash{$\Pi^\perp=\nabla_T^\dagger \nabla_T$} for any spanning tree 
$T$ of $G$. Then, exactly the same arguments as in
\citet{wang2016trend} produce the MSE bound
\begin{equation*}
\|\htheta_G - \theta_0\|_n^2 = O_\P\bigg( 
\frac{M(\nabla_T) \sqrt{\log{n}}}{n} 
\|\nabla_G \theta_0\|_1 \bigg),
\end{equation*}
where \smash{$M(\nabla_T)$} is the maximum $\ell_2$ norm among the
columns of \smash{$\nabla_T^\dagger$}.  We show below, using  
Lemma \ref{lem:tree_edge_incidence_inverse}, that
\smash{$M(\nabla_T) \leq \sqrt{n}$}, and this gives the desired MSE
rate.   

For any $b \in \R^{n-1}$, we may characterize \smash{$\nabla_T^\dagger
  b$} as the unique solution $x \in \R^n$ to the linear system  
\[
\nabla_T x = b,
\]
such that $\one^\top x = 0$, i.e., the unique solution to the linear
system 
\[
\left( \begin{array}{c}
    e_1^\top  \\
    \nabla_T
  \end{array} \right) x = 
\left( \begin{array}{c}
    a  \\
    b
  \end{array} \right),
\]
for a value of $a \in \R$ such that $\one^\top x = 0$. By Lemma 
\ref{lem:tree_edge_incidence_inverse}, we may write
\[
x = A_T \left( \begin{array}{c}
    a  \\
    b
  \end{array} \right),
\]
so that the constraint \smash{$0= \one^\top x = na + \one^\top
  (A_T)_{\cdot, 2:n} b$} gives \smash{$a=-(\one/n)^\top 
  (A_T)_{\cdot,2:n}  b$}, and 
\[
x = (I-\one\one^\top/n) (A_T)_{\cdot, 2:n} b.
\]
Evaluating this across $b=e_1,\ldots,e_n$, we find that the 
maximum $\ell_2$ norm of columns of \smash{$\nabla_T^\dagger$} is
bounded by the maximum $\ell_2$ norm of columns of
\smash{$(A_T)_{\cdot,2:n}$}, which, from the definition in 
\eqref{eq:tree_edge_incidence_inverse}, is at most $\sqrt{n}$. 
\end{proof}

\subsection{Proof of Theorem \ref{thm:bv_minimax}}
\label{app:bv_minimax}

We first present two preliminary lemmas.

\begin{lemma}
\label{lem:part_lower}
Let \smash{$S_1, \ldots, S_m$} be a partition of the nodes of $G$ such
that the total number of edges with ends in distinct elements of the 
partition is at most $s$. Let \smash{$k \leq \min_{i=1,\ldots,m} \,
  |S_i|$}.  Then    
\[
\inf_{\htheta} \; \sup_{\theta_0 \in \BV_G(t)}  \; 
\E \|\htheta - \theta_0 \|_2^2 \geq \frac{k m t^2}{4\sigma^2 s^2}  
\exp\bigg( - \frac{k t^2}{\sigma^2 s^2} \bigg).   
\]
\end{lemma}

\begin{proof}
For each $\eta \in \{-1,1\}^m$, define 
\[
\theta_\eta =  
\frac{\delta}{2}\sum_{i=1}^m \eta_i \frac{1_{S_i}}{\sqrt{|S_i|}},  
\] 
where $\delta>0$ will be specified shortly. Also define the   
class \smash{$\cP =\{ N(\theta_\eta,\sigma^2 I): \eta \in
  \{-1,1\}^m\}$}.   
Note that \smash{$\|\nabla_G \theta_\eta\|_1 \leq \delta  
   s/\sqrt{k}$}, so 
 to embed $\cP$ into the class 
 \smash{$\{ N(\theta,\sigma^2 I) : \theta \in \BV_G(t)\}$}, we set   
 \smash{$\delta = t \sqrt{k}/s$}.

Let $\eta,\eta^{\prime} \in \{-1,1\}^m$ differ in only one
coordinate. Then the KL divergence between the corresponding induced
measures in $\cP$ is \smash{$\|\theta_\eta -
  \theta_{\eta^{\prime}}\|_2^2/\sigma^2 \leq \delta^2/\sigma^2$}.  
Hence by Assouad's Lemma \citep{yu1997assouad}, and a well-known lower
bound on the affinity between probability measures in terms of KL
divergence,   
\[
\inf_{\htheta} \; \sup_{\theta_0 \in \BV_G(t)}  \; 
\E \|\htheta - \theta_0 \|_2^2 \geq \frac{\delta^2 m}{4\sigma^2}
\exp\bigg(-\frac{\delta^2}{\sigma^2}\bigg).   
\]
The result follows by plugging in the specified value for $\delta$. 
\end{proof}

\begin{lemma}
\label{lem:part_tree}
Let $G$ be a tree with maximum degree $d_{\mathrm{max}}$, and
$k \in \{1,\ldots,n\}$ be arbitrary. Then there exists a partition as
in Lemma \ref{lem:part_lower}, $s=m-1$, and  
\[
k \leq \min_{i=1,\ldots,m} \; |S_i| \leq k(d_{\mathrm{max}}+1).
\]
\end{lemma}

\begin{proof}
Our proof proceeds inductively. We begin by constructing
$S_1^{\prime}$,   
the smallest subtree among all those having size at least $k$, and
generated by a cut of size 1 (i.e., separated from the graph by the
removal of 1 edge). Note that \smash{$|S_1^{\prime}| \leq k
  d_{\mathrm{max}}$}, because if not then $S_1^{\prime}$ has at 
least $k$ internal nodes, and we can remove its root to produce
another subtree whose size is smaller but still at least $k$.

For the inductive step, assume 
$S_1^{\prime},\ldots, S_\ell^{\prime}$ have been constructed.   
We consider two cases. (For a subgraph $G^{\prime}$ of $G$, we
denote by $G - G^{\prime}$ the complement subgraph, given by removing
all nodes in $G^{\prime}$, and all edges incident to a node in 
$G^{\prime}$.) 

\smallskip\smallskip
\noindent
{\bf Case 1.} If \smash{$|G  - \cup_{i=1}^{\ell} S_i^{\prime}| > k$},
then we construct $S_{\ell+1}^{\prime}$, the smallest subtree of 
\smash{$G - \cup_{i=1}^{l} S_i^{\prime}$} among all those having size
at least $k$, and generated by a cut of size 1.  As before, we obtain 
that \smash{$|S_{\ell+1}^{\prime}| \leq kd_{\mathrm{max}}$}.  
 
\smallskip\smallskip
\noindent
{\bf Case 2.} If \smash{$|G  - \cup_{i=1}^{\ell} S_i^{\prime}| \leq
  k$}, then the process is stopped.  We define $S_i=S_i^{\prime}$, 
$i=1,\ldots,\ell-1$, as well as 
\smash{$S_\ell=S_\ell^{\prime} \cup (G-\cup_{i=1}^{\ell}
  S_i^{\prime})$}. With $m=\ell$, the result follows.
\end{proof}

We now demonstrate a more precise characterization of the lower bound
in Theorem \ref{thm:bv_minimax}, from which the result in the theorem
can be derived.

\begin{theorem}
\label{thm:bv_lower}
Let $G$ be a tree with maximum degree $d_{\mathrm{max}}$. Then
\[
\inf_{\htheta} \; \sup_{\theta_0 \in \BV_G(t)} \; 
\E \| \htheta - \theta_0 \|_2^2 \geq \frac{t^2}{4 e \sigma^2 n}  
\Bigg(\bigg(\frac{\sigma n}{2t(d_{\max} + 1)}\bigg)^{2/3} -
 1\Bigg)^2.  
 \]
\end{theorem}

\begin{proof}
Set $s = m-1$ and 
\[
k = \floor*{\bigg(\frac{\sigma n}{2t(d_{\mathrm{max}} +  
    1)}\bigg)^{2/3}}.   
\] 
By Lemmas \ref{lem:part_lower} and \ref{lem:part_tree},
\begin{align*}
\inf_{\htheta} \; \sup_{\theta_0 \in \BV_G(t)} \; 
\E \| \htheta - \theta_0 \|_2^2  &\geq
 \frac{km t^2}{4\sigma^2 s^2}
\exp\bigg( - \frac{kt^2}{\sigma^2 s^2} \bigg) \\   
&\geq \frac{kt^2}{4\sigma^2 m}
\exp\bigg( - \frac{k t^2}{\sigma^2 (m-1)^2} \bigg) \\  
&\geq \frac{k t^2}{4\sigma^2 m}
\exp\bigg( - \frac{t^2 k^3(d_{\mathrm{max}} 
  +1)^2}{\sigma^2 n^2} \frac{m^2}{(m-1)^2} \bigg) \\  
&\geq \frac{k t^2}{4\sigma^2 n}
\exp\bigg( - \frac{4t^2 k^3(d_{\mathrm{max}}  
  +1)^2}{\sigma^2 n^2} \bigg) \\ 
&\geq \frac{k^2 t^2}{4\sigma^2 n} \exp(-1).
\end{align*}
In the above, the third line uses \smash{$n/m \leq k
  d_{\mathrm{max}}$} as given by Lemma \ref{lem:part_tree}, the
fourth line simply uses $m \leq n$ and $m^2/(m-1)^2 \leq 4$ (as $m
\geq 2$), and the last line uses the definition of $k$.  Thus, because 
\[
k \geq \bigg(\frac{\sigma n}{2t(d_{\max} + 1)}\bigg)^{2/3} - 1,
\]
we have established the desired result.
\end{proof}

\subsection{Proof of Theorem \ref{thm:bd_minimax}}
\label{app:bd_minimax}

First we establish that, as $G$ is a tree, the number of nodes of
degree at most $2$ is at least $n/2$. Denote by $d_i$ be the 
degree of the node $i$, for each $i=1,\ldots,n$. Then
\[
2(n-1) = \sum_{i=1}^n d_i =  \sum_{i \,:\, d_i \leq 2} d_i + 
  \sum_{i \,:\, d_i \geq 3} d_i \geq  |\{i : d_i \leq 2\}| 
  + 3 |\{ i: d_i \geq  3 \}| = 3n - 2  |\{ i: d_i \leq 2 \}|.    
\]
Hence, rearranging, we find that 
\smash{$|\{ i: d_i \leq 2 \}| \geq n/2+1$}.  

Let $\cI = \{ i: d_i \leq 2 \}$ so that $|\cI| \geq \lceil n/2 \rceil$  
and stipulate that $|\cI|$ is even without loss of generality. 
Let $k$ be the largest even number such that $k \leq s/2$. Define  
\[
\cB = \{ z \in \R^n: z_{\cI} \in \{-1,0,+1\}^{|\cI|}, \; z_{\cI^c} =
0, \; \| z \|_0 = k\}.
\]
Note that by construction \smash{$\cB \subseteq \BD_G(s)$}.

Assume $s \leq n/6$.  Then this implies 
$k/2 \leq n/6 \leq |\cI|/3$.  By Lemma 4 in
\cite{raskutti2011minimax}, there exists 
\smash{$\tilde\cB \subseteq \cB$} such that 
\[
\log|\tilde\cB| \geq \frac{k}{2} 
\log \bigg( \frac{|\cI|-k}{k/2} \bigg),
\] 
and $\|z - z'\|_2^2 \geq k/2$ for all \smash{$z,z' \in \tilde\cB$}. 
Defining \smash{$\cB_0 = 2 \delta \tilde\cB$},
for $\delta > 0$ to be specified shortly, we now have $\|z - 
z'\|_2^2 \geq 2 \delta^2 k$ for all $z,z' \in \cB_0$. 

For $\theta \in \cB_0$, let us consider comparing the measure
$P_\theta = N(\theta,\sigma^2 I)$ against $P_0 = N(0,\sigma^2 I)$: the
KL divergence between these two satisfies $K(P_\theta || P_0) = 
\|\theta\|_2^2/\sigma^2 = 2\delta^2k/\sigma^2$.  Let 
\smash{$\delta = \sqrt{ \alpha\sigma^2/(2k) \log|\cB_0| }$}, for a
parameter $\alpha < 1/8$ that we will specify later.  We have 
\[
\frac{1}{|\cB_0|} \sum_{\theta \in \cB_0}  K(P_\theta || P_0)
\leq \alpha \log |\cB_0|. 
\]
Hence by Theorem 2.5 in \cite{tsybakov2009introduction}, 
\begin{equation}
\label{eq:tsybakov_bd}
\inf_{\htheta} \; \sup_{\theta_0 \in \BD_G(s)} \; 
\P( \| \htheta - \theta_0 \|_2^2 \geq \delta^2 k) \geq
\frac{\sqrt{|\cB_0|}}{1 + \sqrt{|\cB_0|}} \Bigg( 1 - 2\alpha -
  \sqrt{\frac{2\alpha}{\log |\cB_0|}} \Bigg).  
\end{equation}
It holds that
\[
\delta^2 k = \frac{\alpha\sigma^2}{2} \log|\cB_0| 
\geq \frac{\alpha\sigma^2 k}{4} \log 
\frac{|\cI| - k}{k/2} \geq C \sigma^2 s 
\log\bigg(\frac{n}{s}\bigg),    
\]
for some constant $C>0$ depending on $\alpha$ alone.  Moreover, the
right-hand side in \eqref{eq:tsybakov_bd} can be lower bounded by
(say) $1/4$ by  taking
$\alpha$ to be small enough and assuming $n/s$ is large enough.
Thus we have established 
\[
\inf_{\htheta} \; \sup_{\theta_0 \in \BD_G(s)} \; 
\P\Bigg( \| \htheta - \theta_0 \|_2^2 \geq  
C \sigma^2 s 
\log\bigg(\frac{n}{s}\bigg) \Bigg) \geq \frac{1}{4},  
\] 
and the result follows by Markov's inequality.
	
\bibliographystyle{abbrvnat}
\bibliography{graphfused}

\end{document}